\documentclass[11pt]{article}

\usepackage{preprint}
\usepackage{graphicx}
\usepackage{agafonov}
\usepackage{amssymb}
\usepackage{array,booktabs,makecell,colortbl,arydshln}
\usepackage{algorithm}
\usepackage[noend]{algorithmic}
\usepackage[textsize=footnotesize]{todonotes}

\usepackage{xurl}
\usepackage{hyperref}
\hypersetup{
  pdftitle={Matching the Lower Bounds: Stochastic Contracting Cubic Newton and Its Optimal Acceleration},
  pdfauthor={Artem Agafonov, Aslan Ashabokov, Dmitry Kamzolov, Alexander D'yakonov, Alexander Gasnikov, Martin Tak\'a\v{c}}
}

\newsavebox{\contractingfwbox}
\newsavebox{\contractingdomainbox}

\newtheorem{theorem}{Theorem}[section]
\newtheorem{lemma}[theorem]{Lemma}
\newtheorem{proposition}[theorem]{Proposition}

\newtheorem{assumption}[theorem]{Assumption}

\theoremstyle{plain}

\begin{document}

\title{Matching the Lower Bounds: Stochastic Contracting Cubic Newton and Its Optimal Acceleration}
\author{%
\begin{tabular}{c@{\hspace{2em}}c@{\hspace{2em}}c}
Artem Agafonov$^{1}$ & Aslan Ashabokov$^{2}$ & Dmitry Kamzolov \\[6pt]
Alexander D'yakonov & Alexander Gasnikov$^{2}$ & Martin Tak\'a\v{c}$^{1}$
\end{tabular}\\[10pt]
{\normalsize $^{1}$MBZUAI\qquad $^{2}$MIRAI}}
\maketitle

\begin{abstract}
  We study second-order methods for convex stochastic optimization, where gradients and Hessians are available only through stochastic estimates with variances $\sigma_1^2$ and $\sigma_2^2$, respectively. First, we propose the Stochastic Contracting Cubic Newton method. At each iteration, it minimizes a cubic model with additional quadratic regularization and then contracts the step toward the current point. After $T$ iterations, the method achieves the expected convergence rate $\mathcal{O}(\sigma_1/\sqrt{T}+\sigma_2/T+1/T^2)$. Building on this construction, we develop an accelerated variant achieving $\mathcal{O}(\sigma_1/\sqrt{T}+\sigma_2/T^2+1/T^{7/2})$, matching the known lower bounds of~\citet{agafonov2024advancing} in all three terms.
\end{abstract}

\section{Introduction}
\label{sec:introduction}
In this paper, we consider the problem 
\begin{equation}
  \label{eq:problem}
  \min_{x\in\XCal} f(x),
\end{equation}
where $f:\R^d\to\R$ is convex and twice continuously differentiable, and $\XCal\subseteq\R^d$ is a nonempty closed convex set of diameter $D>0$. We assume that a solution $x^{\ast}\in \XCal$ exists and we denote $f^{\ast} = f(x^{\ast})$.

In many applications, exact gradients and Hessians of $f$ are expensive or unavailable. A typical example is machine learning, where the objective is a finite sum or an expectation over data, and only stochastic estimates of its derivatives can be computed at a reasonable cost. This motivates optimization methods that work with inexact and stochastic derivatives.

\textbf{First-order methods.} 
For convex functions with $L_1$-Lipschitz continuous gradients and stochastic gradients with variance $\sigma_1^2$,~\citet{nemirovskiĭ1983problem} established the first-order complexity lower bound
$\Omega \ls
\tfrac{\sigma_1 R}{\sqrt T}
+
\tfrac{L_1 R^2}{T^2}
\rs$.
The accelerated gradient method of~\citet{nesterov1983method} matches the second term in the exact case. Stochastic approximation methods~\citep{nemirovski2009robust} attain the optimal dependence on the noise, but their deterministic rate is non-accelerated. Matching both terms with a single method remained open for nearly three decades. The gap was closed by the accelerated stochastic approximation method of~\citet{lan2012optimal}.

\textbf{Second-order methods.} In contrast, the theory of stochastic second-order optimization is less complete. The lower bound
$\Omega \ls
\tfrac{\sigma_1 R}{\sqrt T}
+
\tfrac{\sigma_2 R^2}{T^{2}}
+
\tfrac{L_2 R^3}{T^{7/2}}
\rs$ 
was established only recently~\citep{agafonov2024advancing}. Existing methods match this lower bound partially. The adaptive Extra-Newton method of~\citet{antonakopoulos2022extra} achieves the optimal dependence on gradient noise. Stochastic accelerated cubic Newton of~\citet{agafonov2024advancing} matches both stochastic terms in the lower bound, while its rate in the exact case is $\cO(T^{-3})$. More recently,~\citet{chen2026optimal} showed that optimal acceleration is possible with inexact Hessians but exact gradients. These rates are summarized in Table~\ref{tab:convex_comparison}.

Despite this progress, \emph{developing an optimal stochastic second-order method remains an open problem}.

Even more surprisingly, to the best of our knowledge, there is no basic (non-accelerated) stochastic second-order method that attains the optimal gradient noise term.

\begin{table}[t]
  \centering
  \small
  \definecolor{optimalcell}{RGB}{240,253,241}
  \caption{Convergence rates for convex inexact and stochastic second-order methods.}
  \label{tab:convex_comparison}
  \setlength{\tabcolsep}{4pt}
  \setlength{\arrayrulewidth}{0.4pt}
  \setlength{\aboverulesep}{0pt}
  \setlength{\belowrulesep}{0pt}
  \setlength{\dashlinedash}{0.2pt}
  \setlength{\dashlinegap}{1.5pt}
  \renewcommand{\arraystretch}{1.15}
  \makeatletter
  \let\comparison@BTnormal\@BTnormal
  \def\@BTnormal{%
    \adl@hline\adl@connect\@thisrulewidth
    \global\adl@finaldepth\z@
    \comparison@BTnormal}
  \makeatother
  \begin{tabular}{@{}
    >{\centering\arraybackslash}m{0.47\textwidth}:
    >{\centering\arraybackslash}m{0.11\textwidth}:
    >{\centering\arraybackslash}m{0.11\textwidth}:
    >{\centering\arraybackslash}m{0.11\textwidth}:
    >{\centering\arraybackslash}m{\dimexpr0.18\textwidth-8\tabcolsep-4\arrayrulewidth\relax}
  @{}}
    \toprule
    Method
      & \makecell{Gradient\\error term}
      & \makecell{Hessian\\error term}
      & Exact term
      & $\XCal$ \\
    \midrule
    \makecell{Inexact Cubic Newton ~\citep{agafonov2023inexact}}
      & $\cO\ls\delta_1\bar R\rs$
      & $\cO\ls\tfrac{\delta_2\bar R^2}{T}\rs$
      & $\cO\ls\tfrac{L_2\bar R^3}{T^2}\rs$
      & $\R^d$ \\
    \hdashline
    \makecell{Contracting-Domain Newton\\{~\citep{doikov2020convex}}~(Proposition~\ref{prop:weak_convergence_guarantee})}
      & $\cO\ls\sigma_1D\rs$
      & $\cO\ls\tfrac{\sigma_2D^2}{T}\rs$
      & $\cO\ls\tfrac{L_2D^3}{T^2}\rs$
      & compact \\
    \hdashline
    \makecell{Accelerated Inexact Cubic Newton\\{~\citep{agafonov2023inexact}}}
      & $\cO\ls\delta_1 \bar{D}\rs$
      & \cellcolor{optimalcell} $\cO\ls\tfrac{\delta_2R^2}{T^2}\rs$
      & $\cO\ls\tfrac{L_2R^3}{T^3}\rs$
      & $\R^d$ \\
    \hdashline
    \makecell{Extra-Newton~\citep{antonakopoulos2022extra}}
      & \cellcolor{optimalcell} $\cO\ls\tfrac{\sigma_1D}{\sqrt T}\rs$
      & $\cO\ls\tfrac{\sigma_2D^2}{T^{3/2}}\rs$
      & $\cO\ls\tfrac{L_2D^3}{T^3}\rs$
      & compact \\
    \hdashline
    \makecell{Accelerated Stochastic Cubic Newton\\{~\citep{agafonov2024advancing}}}
      & \cellcolor{optimalcell} $\cO\ls\tfrac{\sigma_1R}{\sqrt T}\rs$
      & \cellcolor{optimalcell} $\cO\ls\tfrac{\sigma_2R^2}{T^2}\rs$
      & $\cO\ls\tfrac{L_2R^3}{T^3}\rs$
      & $\R^d$ \\
    \hdashline
    \makecell{AINE{~\citep{chen2026optimal}}}
      & \xmark
      & \cellcolor{optimalcell} $\cO\ls\tfrac{\delta_2R^2}{T^2}\rs$
      & \cellcolor{optimalcell} $\cO\ls\tfrac{L_2R^3}{T^{7/2}}\rs$
      & $\R^d$ \\
    \midrule
    \makecell{Lower bound~\citep{agafonov2024advancing}}
      & \cellcolor{optimalcell} $\Omega\ls\tfrac{\sigma_1R}{\sqrt T}\rs$
      & \cellcolor{optimalcell} $\Omega\ls\tfrac{\sigma_2R^2}{T^2}\rs$
      & \cellcolor{optimalcell} $\Omega\ls\tfrac{L_2R^3}{T^{7/2}}\rs$
      & $\R^d$ \\
    \midrule
    \makecell{Stochastic Contracting Cubic Newton\\(This work, Theorem~\ref{thm:basic_convergence})}
      & \cellcolor{optimalcell} $\cO\ls\tfrac{\sigma_1D}{\sqrt T}\rs$
      & $\cO\ls\tfrac{\sigma_2D^2}{T}\rs$
      & $\cO\ls\tfrac{L_2D^3}{T^2}\rs$
      & compact \\
    \hdashline
    \makecell{Accelerated Stochastic Contracting Cubic Newton\\(This work, Theorem~\ref{thm:optimal_convergence})}
      & \cellcolor{optimalcell} $\cO\ls\tfrac{\sigma_1D}{\sqrt T}\rs$
      & \cellcolor{optimalcell} $\cO\ls\tfrac{\sigma_2D^2}{T^2}\rs$
      & \cellcolor{optimalcell} $\cO\ls\tfrac{L_2D^3}{T^{7/2}}\rs$
      & compact \\
    \bottomrule
  \end{tabular}
  \par\smallskip
  \begin{minipage}{\textwidth}
    \scriptsize
    Green cells mark the dependence on $T$ matching the corresponding lower-bound term.
    $R=\norm{x_0-x^*}$, $D=\operatorname{diam}(\XCal)$. For methods from~\citep{agafonov2023inexact} $\bar R$ is the diameter of the level set, $\bar{D}$ bounds the distances between iterates and $x^*$.
    The parameters $\delta_1,\delta_2$ denote gradient and Hessian inexactness $\norm{g(x)-\g(x)}\leq\delta_1$ and $\norm{H(x)-\h(x)}\leq\delta_2$.
  \end{minipage}
\end{table}

\textbf{Contributions.} Our main result is the Accelerated Stochastic Contracting Cubic Newton method. For convex optimization over compact set it achieves the convergence rate
 $\cO \ls
\tfrac{\sigma_1 D}{\sqrt T}
+
\tfrac{\sigma_2 D^2}{T^{2}}
+
\tfrac{L_2 D^3}{T^{7/2}}
\rs$
matching the lower bound.

We summarize our contributions as follows:
\begin{itemize}
  \item \textbf{Basic method.}
    We introduce the Stochastic Contracting Cubic Newton method 
    (Algorithm~\ref{alg:basic_method}), which attains the optimal dependence on gradient noise,
    $\cO(\sigma_1 D/\sqrt{T})$.
    To the best of our knowledge, it is the first non-accelerated globally convergent stochastic
    second-order method with this dependence.
    The method is obtained by augmenting the Contracting-Domain Newton method
    of~\citet{doikov2020convex} with an additional quadratic regularization centered at the
    previous auxiliary point.
    \item \textbf{Reduction to a first-order scheme.} We show that  removing the Hessian and the cubic regularization from Stochastic Contracting Cubic Newton reduces it to the projected stochastic gradient method with averaging. We prove the convergence rate  $\cO\ls\tfrac{\sigma_1 D}{\sqrt T}+\tfrac{L_1 D^2}{T}\rs$ for this first-order scheme. In contrast, removing the Hessian term from the contracting domain Newton method
    of~\citet{doikov2020convex} leads to a naive stochastic Frank-Wolfe update, which may fail to converge under gradient noise.
    This reduction provides a natural explanation for why the additional quadratic regularization
    yields the optimal $\cO(\sigma_1 D/\sqrt{T})$ dependence on gradient noise.
  \item \textbf{Accelerated method}. We propose the Accelerated Stochastic Contracting Cubic Newton method (Algorithm~\ref{alg:optimal_method}), which combines
    ideas from the Contracting-Domain Newton method of~\citet{doikov2020convex},
    the optimal Monteiro-Svaiter acceleration of~\citet{carmon2022optimal},
    and the accelerated stochastic second-order scheme of~\citet{agafonov2024advancing}.
    We show that the resulting method simultaneously attains the optimal dependence on
    gradient noise, Hessian noise, and second-order smoothness.
\end{itemize}

\section{Related Work}

Classical Newton method is a standard tool in smooth minimization, known for its fast local quadratic convergence~\citep{bennett1916newton,kantorovich1949newton,nesterov2018lectures}. Its variants are used, for example, as logistic regression solvers
in the \texttt{scikit-learn} library~\citep{pedregosa2011scikit}.

A major advance in the development of globally convergent Newton-type
methods was cubic regularization~\citep{nesterov2006cubic}. For convex objectives with Lipschitz-continuous Hessians, this method
achieves a global convergence rate of $\cO(L_2 \bar{R}^3/T^2)$, matching
the dependence on $T$ of accelerated gradient
descent~\citep{nesterov1983method}. Later, it was accelerated to $\cO(L_2R^3/T^3)$ by~\citet{nesterov2008accelerating}. Further,~\citet{monteiro2013accelerated} proposed the accelerated
Newton Proximal Extragradient,  which achieves lower complexity bound $\Omega \ls{L_2R^3}/{T^{7/2}}\rs$ \citep{pmlr-v75-agarwal18a,arjevani2019oracle} up to a logarithmic factor. This gap was closed
by~\citet{kovalev2022first} and~\citet{carmon2022optimal}. Under strong convexity, cubic regularized Newton also achieves global superlinear convergence~\citep{kamzolov2025optami}.

Another line of work reduces the computational cost of Newton-type
methods by avoiding cubic subproblems.
\citet{mishchenko2023regularized,doikov2023gradient} use
gradient-dependent quadratic regularization,
while~\citet{hanzely2022damped} introduce an explicit damping rule
for the Newton step under semi-strong self-concordance. 
\citet{doikov2026minimizing} further establishes global linear
convergence of the gradient-regularized Newton method
for quasi-self-concordant functions. For optimization over compact domains,~\citet{doikov2020convex}
proposed Contracting-Domain Newton methods that solve quadratic
subproblems and achieve a global convergence rate of $\cO(L_2D^3/T^2)$.
They also extended these methods to finite-sum minimization
using increasing batch sizes and variance reduction.

Another line of work studies methods with inexact or stochastic derivatives.
\citet{ghadimi2017second} proposed basic and accelerated cubic Newton
methods with inexact Hessians.
\citet{agafonov2023inexact} extended this approach to inexact gradients
and higher-order derivatives.
Accelerated methods with stochastic gradients and Hessians were
developed by~\citet{antonakopoulos2022extra} and~\citet{agafonov2024advancing}. 
\citet{chayti2025improving} proposed stochastic cubic Newton with momentum and, for convex objectives with bounded iterates, obtained a gradient-noise term of order $T^{-2/5}$, which is slower than the optimal $T^{-1/2}$. Recently,~\citet{chen2026optimal} proposed the Accelerated Inexact
Newton Extragradient method, which achieves the optimal rate with exact gradients and inexact Hessians. Inexact Hessians have also been studied under generalized 
smoothness assumptions~\citep{semenov2026gradient,agafonov2026optimal}.

Inexact Hessians can reduce the computational and communication costs
of second-order methods.~\citet{kamzolov2026accelerated} use inexact cubic regularization to develop
Quasi Newton methods with global convergence guarantees. Lazy Hessian updates reduce computational costs by reusing Hessians
across several iterations~\citep{doikov2023second,zhang2026nearoptimalconvexoptimizationlazy}. In distributed optimization, local Hessian
approximations and approximate consensus help control communication
costs~\citep{zhang2015disco,agafonov2021accelerated,dvurechensky2022hyperfast,agafonov2026decentralized}.

\section{Preliminaries}
\label{sec:preliminaries}

We use Euclidean norms for vectors and spectral norms for matrices. Throughout the paper, we assume that the Hessian of $f$ is $L_2$-Lipschitz continuous.
\begin{assumption}
  \label{as:lip_hess}
  The function $f \in \mathcal{C}^2$ satisfies
  \begin{equation}
    \label{eq:lip_hess}
    \norm{\h(x)-\h(y)} \leq L_2\norm{x-y},
    \quad x,y\in\XCal.
  \end{equation}
\end{assumption}
Under Assumption~\ref{as:lip_hess}, the following standard Taylor remainder bounds hold.
\begin{lemma}[{\citep[Lemma~1.2.4]{nesterov2018lectures}}]
  \label{lem:taylor_remainder}
  Let $f$ be twice continuously differentiable and satisfy~\eqref{eq:lip_hess}. Then, for every $x,y\in\XCal$,
  \begin{equation*}
    \begin{aligned}
      \bigl|f(y)&-f(x)-\la\g(x),y-x\ra 
      -\tfrac{1}{2}\la\h(x)(y-x),y-x\ra\bigr|
      \leq\tfrac{L_2}{6}\norm{y-x}^3.
    \end{aligned}
  \end{equation*}
  \begin{equation*}
    \norm{\g(y)-\g(x)-\h(x)(y-x)}
    \leq\tfrac{L_2}{2}\norm{y-x}^2.
  \end{equation*}
\end{lemma}

We assume that stochastic gradients satisfy the following condition.
\begin{assumption}
  \label{as:stoch_grad}
  For every $x\in\XCal$, a stochastic gradient $g(x,\xi)\in\R^d$ satisfies
  \begin{equation}
    \label{eq:stoch_grad}
    \E\lp\norm{g(x,\xi)-\g(x)}^2\mid x\rp \leq \sigma_1^2.
  \end{equation}
\end{assumption}
Some of our results also require unbiasedness of the gradient.
\begin{assumption}
  \label{as:unbiased_grad}
  For every $x\in\XCal$, a stochastic gradient $g(x,\xi)\in\R^d$ satisfies:
  \begin{equation}
    \label{eq:unbiased_grad}
    \E[g(x,\xi)\mid x] = \g(x).
  \end{equation}
\end{assumption}

We make the following assumption on stochastic Hessians.
\begin{assumption}
  \label{as:stoch_hess}
  For every $x\in\XCal$, a stochastic Hessian $H(x,\xi)\in\R^{d\times d}$ is symmetric and positive semidefinite almost surely, and satisfies
  \begin{equation}
    \label{eq:stoch_hess}
    \E\lp\norm{H(x,\xi)-\h(x)}^2\mid x\rp \leq \sigma_2^2.
  \end{equation}
\end{assumption}

We do not require the Hessian estimates to be unbiased. In many applications, stochastic Hessian estimates are naturally symmetric and positive semidefinite (e.g., mini-batch Hessians of convex loss functions). At each iteration, the random sample $\xi$ is drawn independently of the past and may be shared between the gradient and Hessian estimates. We suppress the dependence on $\xi$ and write $g(x)$ and $H(x)$ for the sampled estimates.

\subsection{Contracting Newton Method}

In this subsection, we review the Contracting-Domain Newton
method~\citep[Algorithms~1,2]{doikov2020convex}, which forms the basis
for the algorithms proposed in this work.

At iteration $t\geq0$, given $g_t=g(x_t)$ and $H_t=H(x_t)$,
define the quadratic model
\begin{equation}
  \label{eq:quadratic_model}
  \phi_{t, \alpha} (u) = \la g_t, u - x_t \ra + \tfrac{\alpha}{2} \la H_t  (u - x_t), u - x_t \ra
\end{equation}
for some $0 < \alpha \leq 1$. Starting from $x_0\in\XCal$, the method performs
\begin{equation}
  \label{eq:nikita_cn}
  u_{t+1} = \argmin_{u \in \XCal} \phi_{t, \alpha_{t}} (u),~x_{t+1} = x_t + \alpha_t (u_{t+1} - x_t).
\end{equation}
With exact derivatives \citet[Theorem 1]{doikov2020convex} establish the global rate $\cO(L_2D^3/t^2)$ for $\alpha_t = \cO(1/t)$, matching cubic Newton~\citep{nesterov2006cubic}.

One interpretation of such scheme is a second-order Frank-Wolfe-type method~\citep{frank1956algorithm}, in which the linear first-order model is replaced by the quadratic one.  Equivalently, one iteration of the procedure~\eqref{eq:nikita_cn} can be considered as a minimizer of the model $\phi_{t, 1} (u)$ over a contracted domain $\YCal_t = (1 - \alpha_t)x_{t} + \alpha_t \XCal$:
\begin{equation*}
  x_{t+1} = \argmin_{u \in \YCal_t} \lb \la g_t, u - x_{t} \ra + \tfrac{1}{2} \la H_t  (u - x_{t}), u - x_{t} \ra \rb.
\end{equation*}
For finite-sum problems,~\citet[Theorem 5]{doikov2020convex} give sufficient condition on gradient and Hessian batch sizes to recover $\cO(1/t^2)$ rate. The following proposition shows that these methods can be directly applied for stochastic optimization. The resulting rate matches the convergence of inexact cubic Newton~\citep[Theorem 2]{agafonov2023inexact}.
\begin{proposition}
  \label{prop:weak_convergence_guarantee}
  Let $f$ be convex. 
  Let Assumptions~\ref{as:lip_hess},~\ref{as:stoch_grad}, and~\ref{as:stoch_hess} hold. 
  Then the procedure~\eqref{eq:nikita_cn} with $\alpha_t = 3/(t+3)$ satisfies, for every  $T\geq1$,
  \begin{equation}
    \label{eq:weak_convergence_guarantee}
    \E\lp f(x_T)-f^*\rp
    =\cO\ls\sigma_1D +\tfrac{\sigma_2D^2}{T} + \tfrac{L_2 D^3}{T^2}\rs.
  \end{equation}
\end{proposition}

Importantly, the bound above does not require unbiasedness of the stochastic derivatives. However, even if the gradient estimator is unbiased, there is a concrete obstacle in the analysis that prevents us from directly obtaining the desired $\cO(\sigma_1D/\sqrt{T})$ term.

Denote $\Delta_t = f(x_t) - f^*$. Let us introduce a one-step progress (see Lemma~\ref{lem:basic_one_step}) 
\begin{equation}
  \label{eq:nikita_one_step}
  \begin{aligned}
    \Delta_{t+1}
    \leq{} (1-\alpha_{t})\Delta_{t} 
    + \alpha_{t}\la\g(x_{t})-g_{t},u_{t+1}-x^*\ra
    +\alpha_{t}^2D^2\norm{H_{t}-\h(x_{t})}
    +\tfrac{L_2}{3}\alpha_{t}^3D^3.
  \end{aligned}
\end{equation}
Even when the gradient estimator is unbiased, the gradient error term does not disappear, since $u_{t+1}$ depends on $g_t$. Instead, applying Cauchy--Schwarz gives 
\begin{equation*}
  \begin{aligned}
    \E \Delta_{t+1}
    \leq{} (1-\alpha_{t})\E \Delta_{t} 
    + \alpha_{t}\sigma_1 D +\alpha_{t}^2 \sigma_2 D^2
    +\tfrac{L_2}{3}\alpha_{t}^3D^3.
  \end{aligned}
\end{equation*}
Telescoping this, we arrive at the rate~\eqref{eq:weak_convergence_guarantee}. In the next section, we provide the direct way to overcome this problem.

\section{Stochastic Contracting Cubic Newton}

We begin with introducing a new model of objective.
Given a realization of stochastic gradient $g_t$ and stochastic Hessian $H_t$, we define
\begin{equation}
  \begin{aligned}
  \label{eq:cubic_model}
  \omega^M_{t}(u) =~ \la g_{t},u-x_{t}\ra 
          +\tfrac{\alpha_t}{2}\la H_{t}(u-x_{t}),u-x_{t}\ra
          + \tfrac{\eta_t
          }{2}\|u - u_{t}\|^2 + \tfrac{M \alpha_t^2}{6}\norm{u-x_{t}}^3 
  \end{aligned}
\end{equation}
for some $\alpha_t,\eta_t$ and $M \geq 0$. The key difference with the quadratic model~\eqref{eq:quadratic_model} and methods of~\citet{doikov2020convex} is an additional quadratic regularization with parameter $\eta_t$. In this section, we choose
regularization parameters
\begin{equation}
   \label{eq:basic_reg}
    \alpha_t = \frac{3}{t+3}, \qquad \eta_t = \frac{\sigma_1 \sqrt{t+1}}{D}.
\end{equation}
We also allow an optional cubic regularization with parameter $M\geq0$, which is not required for our convergence guarantee. Beyond providing a more general formulation, it may also improve the stability of the method outside the bounded-domain setting.

The resulting Stochastic Contracting Cubic Newton method is listed as Algorithm~\ref{alg:basic_method}.
\begin{algorithm}[H]
  \caption{Stochastic Contracting Cubic Newton method}
  \label{alg:basic_method}
  \begin{algorithmic}[1]
    \STATE \textbf{Input:} $x_0 = u_0 \in \XCal$ and $M \geq0$.
    \FOR{$t=0,1,\ldots$}
      \STATE Query $g_{t} = g(x_{t})$ and $H_{t} = H(x_{t})$.
      \STATE Compute $u_{t+1}\in\argmin_{u\in \XCal} \omega_{t}^M (u)$. \label{line:minimize_model}
      \STATE Set $x_{t+1} = (1-\alpha_t)x_{t}+\alpha_t u_{t+1}$. \label{line:contraction_step}
    \ENDFOR
  \end{algorithmic}
\end{algorithm}

\subsection{Convergence Analysis}

The next lemma shows one step progress bound using smoothness, convexity and optimality condition.
\begin{lemma}
  \label{lem:basic_one_step_main}
  Let $f$ be convex and Assumption~\ref{as:lip_hess} hold. Let $M\geq0$. For the iterates of Algorithm~\ref{alg:basic_method}, we have
  \begin{equation*}
    \begin{aligned}
      \Delta_{t+1}
      \leq{}& (1-\alpha_t)\Delta_t 
      +\tfrac{\alpha_t\eta_t}{2}\ls
        \norm{u_t-x^*}^2-\norm{u_{t+1}-x^*}^2 \rs
      -\tfrac{\alpha_t\eta_t}{2}\norm{u_{t+1}-u_t}^2 
       \\
      ~~~~+&~\alpha_t\la\g(x_t)-g_t,u_{t+1}-x^*\ra + \alpha_t^2D^2\norm{H_t-\h(x_t)}
      +\tfrac{\max\lb M,L_2\rb}{3}\alpha_t^3D^3.
    \end{aligned}
  \end{equation*}
\end{lemma}
Comparing with~\eqref{eq:nikita_one_step}, this bound has additional terms involving $\eta_t$, which will allow us to take expectation properly. First,
\begin{equation*}
   \E  \langle \g(x_t)- g_t,  u_{t+1}-x^*\rangle
    = \E \langle\g(x_t)-g_t,u_{t+1}-u_t\rangle,
\end{equation*}
where we used unbiasedness of the gradient.
Therefore, by Young's inequality
\begin{gather*}
    \E [\alpha_t\la\g(x_t)-g_t,u_{t+1}-u_t\ra -\tfrac{\alpha_t\eta_t}{2}\norm{u_{t+1}-u_t}^2] 
    \leq\tfrac{\alpha_t}{2\eta_t}\E\norm{g_t-\g(x_t)}^2
    \leq\tfrac{\sigma_1^2 \alpha_t}{2\eta_t}.
\end{gather*}
Taking expectations of the remaining terms and
using~\eqref{eq:stoch_hess}, we obtain
\begin{equation*}
  \begin{aligned}
  \E\Delta_{t+1}
  \leq{} &(1-\alpha_t)\E\Delta_t+\tfrac{\alpha_t\sigma_1^2}{2\eta_t}
  +\alpha_t^2\sigma_2D^2 + \tfrac{\max\lb M,L_2\rb}{3}\alpha_t^3D^3\\
  &+\tfrac{\alpha_t\eta_t}{2}\ls
    \E\norm{u_t-x^*}^2
    -\E\norm{u_{t+1}-x^*}^2
  \rs.
  \end{aligned}
\end{equation*}

Compared to Contracting-Domain Newton method~\eqref{eq:nikita_cn}, the term
$\alpha_t\sigma_1D$ is replaced by
$\alpha_t\sigma_1^2/(2\eta_t)$ and a difference of squared
distances. In particular,
  \begin{equation*}
  \alpha_t\sigma_1D
  \simeq \frac{\sigma_1D}{t},
  \qquad
  \frac{\alpha_t\sigma_1^2}{2\eta_t}
  \simeq \frac{\sigma_1D}{t^{3/2}}.
  \end{equation*}
The remaining difference of squared distances is controlled by weighted telescoping.

To obtain the final convergence rate, it remains to multiply the inequality above by
$(t+1)(t+2)(t+3)$ and sum over $t=0,\ldots,T-1$.
The regularization parameter $\eta_t$ controls the balance between the two accumulated contributions, yielding the desired
$\cO(\sigma_1D/\sqrt{T})$ dependence. 

We formalize this result in the following theorem, with the full proof given in Appendix~\ref{sec:proof_basic}.
\begin{theorem}
  \label{thm:basic_convergence}
  Let $f$ be convex. Let Assumptions~\ref{as:lip_hess},~\ref{as:stoch_grad},~\ref{as:unbiased_grad}, and~\ref{as:stoch_hess} hold. Then Algorithm~\ref{alg:basic_method} with $M\geq0$ and parameters~\eqref{eq:basic_reg} satisfies, for every $T\geq1$,
  \begin{equation*}
    \E\lp f(x_T)-f^*\rp
    =\cO\ls\tfrac{\sigma_1D}{\sqrt T} +\tfrac{\sigma_2D^2}{T} + \tfrac{\max\lb M,L_2\rb D^3}{T^2}\rs.
  \end{equation*}
\end{theorem}

\subsection{First-order Reduction}
Setting $H_t=0$ and $M=0$ reduces Algorithm~\ref{alg:basic_method} to projected stochastic gradient descent with averaging
\begin{equation}
  \label{eq:FO_reduction}
  \begin{gathered}
    u_{t+1}
    =
    \Pi_{\XCal}\left(u_t-\tfrac{1}{\eta_t}g_t\right),\\
    x_{t+1}
    =
    (1-\alpha_t)x_t+\alpha_tu_{t+1}.
  \end{gathered}
\end{equation}
Therefore, additional quadratic regularizer induces a projected
first-order step with stepsize $1/\eta_t$, while contraction reduces to iterate averaging only. The stochastic gradient is
evaluated at the averaged point $x_t$. The same first-order scheme, with the same parameter dependence
on $t$ up to numerical constants, is studied by~\citet[(18), Theorem~2]{defazio2021power}.

In contrast, setting $H_t=0$ in the classical Contracting-Domain Newton method~\eqref{eq:nikita_cn} of~\citet{doikov2020convex} reduces it to the naive stochastic Frank-Wolfe update, which may fail to converge under nonvanishing gradient noise~\citep{mokhtari2020stochastic}.

Let us further assume that $f$ has $L_1$-Lipschitz gradient
\begin{equation}
  \label{eq:lip_grad}
  \norm{\g(x)-\g(y)} \leq L_1\norm{x-y},
  \quad x,y\in\XCal.
\end{equation}

A direct application of
Theorem~\ref{thm:basic_convergence} with $H_t=0$ gives $\cO\ls\tfrac{\sigma_1D}{\sqrt T} +\tfrac{L_1 D^2}{T} + \tfrac{L_2 D^3}{T^2}\rs$ rate for the scheme~\eqref{eq:FO_reduction} since
$\|\nabla^2f(x)\|\le L_1$. However, the last term is an artifact of using a
second-order Taylor approximation in
Lemma~\ref{lem:basic_one_step_main}.
For the first-order reduction, the standard smoothness
inequality is sufficient.

\begin{proposition}
  \label{prop:first_order_reduction}
Let $f$ be convex and satisfy~\eqref{eq:lip_grad}. Let Assumptions~~\ref{as:stoch_grad} and ~\ref{as:unbiased_grad} hold. 
Then Algorithm~\ref{alg:basic_method} with
$H_t=0$, $M=0$, and parameters~\eqref{eq:basic_reg}
satisfies, for every integer $T\geq1$,
\begin{equation*}
  \E[f(x_T)-f^*] 
  =
  \cO\ls
  \tfrac{\sigma_1D}{\sqrt T}
  +\tfrac{L_1D^2}{T}
  \rs.
\end{equation*}
\end{proposition}

The proof is given in Appendix~\ref{sec:proof_first_order}.

\section{Matching the Lower Bound}
In this section, we develop an accelerated method, which matches the lower bound rate of~\citet{agafonov2024advancing}. Our Accelerated Stochastic Contracting Cubic Newton method is listed as Algorithm~\ref{alg:optimal_method}. It combines the ideas from Algorithm~\ref{alg:basic_method} with damping optinal acceleration of~\citet{carmon2022optimal}.

\begin{algorithm}[H]
  \caption{Accelerated Stochastic Contracting Cubic Newton}
  \label{alg:optimal_method}
  \begin{algorithmic}[1]
    \STATE \textbf{Input:} $w_0=z_0=x_0\in\XCal$, $A_0=0$;
    $M>0$, $M\geq L_2$; parameters~\eqref{eq:acc_params}.
    \FOR{$t=1,2,\ldots$}
      \STATE Set $v_t=(1-\alpha_t)w_{t-1}+\alpha_tz_{t-1}$.
      \STATE Sample $\xi_t$, query $g_t=g(v_t,\xi_t)$, 
      $H_t=H(v_t,\xi_t)$. \label{lin:deriv}
      \STATE Compute $u_t=\argmin_{u\in\XCal}\Omega_t(u)$.
      \label{lin:cubic_sub_acc} \label{lin:opt_model_min}
      \STATE Set $y_t=(1-\alpha_t)w_{t-1}+\alpha_tu_t$.
      \label{lin:contrction_acc} \label{lin:opt_contract}
      \STATE Set $\tau_t=\beta_t+\lambda_t\norm{y_t-v_t}$.
      \label{lin:real_tau}
      \STATE Set $a_t,A_t$ according to~\eqref{eq:A_relation_main}.
      \STATE Sample $\widetilde\xi_t$, query
      $\widetilde g_t=g(y_t,\widetilde\xi_t)$.  \label{lin:unbiased_grad}
      \STATE Set $w_t=\tfrac{A_{t-1}w_{t-1}+a_tu_t}{A_t}$.
      \STATE Compute \label{lin:avg_sub_acc}
      \begin{equation*}
        z_t=\argmin_{z\in\XCal}\left\{\la\widetilde g_t,z\ra
        +\tfrac{\lambda_t}{2a_t}\norm{z-z_{t-1}}^2\right\}.
      \end{equation*}
    \ENDFOR
  \end{algorithmic}
\end{algorithm}

\subsection{The Method}
Given stochastic derivatives $g_t = g(v_t,\xi_t)$ and $H_t = H(v_t,\xi_t)$, we define a cubic model for accelerated method
\begin{equation*}
  \begin{aligned}
    \Omega_t(u) = \la g_t, u - z_{t-1} \ra + \tfrac{\alpha_t}{2} \la H_t (u - z_{t-1}), u - z_{t-1} \ra 
    + \tfrac{\alpha_t \beta_t}{2} \|u - z_{t-1}\|^2 + \tfrac{\lambda_t \alpha_t^2}{3}\|u - z_{t-1}\|^3,
  \end{aligned}
\end{equation*}
where for $M \geq L_2$
\begin{equation}
  \label{eq:acc_params}
  \lambda_t=\tfrac{\sigma_1t^3}{D^2}+\tfrac{\sigma_2t^{3/2}}{D}+2M,~~\alpha_t = \tfrac{4}{t+3},~~\beta_t = \tfrac{\lambda_t D}{t^{3/2}}.
\end{equation}
Lines~\ref{lin:opt_model_min}-\ref{lin:opt_contract} of Algorithm~\ref{alg:optimal_method} perform the contracted cubic Newton step. This construction can equivalently be written as a local model minimized over a contracted domain. Changing variables from $u$ to $y = (1 - \alpha_t) w_{t-1} + \alpha_t u$, we have $\alpha_t (u - z_{t-1}) = y - v_t$. Therefore, the contracting step is equivalent to
\begin{equation*}
  \begin{aligned}
    y_t = \argmin_{y \in \YCal_t} \left \{ \la g_t, y - v_t \ra + \tfrac{1}{2} \la H_t (y - v_t), y - v_t \ra  + \tfrac{\beta_t}{2} \|y - v_t\|^2 + \tfrac{\lambda_t}{3}\|y - v_t\|^3
    \right\},
  \end{aligned}
\end{equation*}
where $\YCal_t = (1 - \alpha_t) w_{t-1} + \alpha_t \XCal$.
In contrast to the non-accelerated Algorithm~\ref{alg:basic_method}, quadratic and cubic regularizers are centered at $v_t$, the same point at which stochastic derivatives are evaluated. 

This makes our subproblem similar to the one used in the accelerated stochastic cubic Newton method of~\citet{agafonov2024advancing}, but restricted to $\YCal_t$. In particular, their quadratic regularization scales as $\sigma_2 + \sigma_1 t^{3/2} / R$, which matches the noise-dependent part of our $\beta_t$~\eqref{eq:acc_params}.

To achieve optimal acceleration we apply the damping mechanism of~\citet{carmon2022optimal}.
With $s_t=\alpha_t(u_t-z_{t-1})$, the optimality condition for $u_t$ gives
\begin{align*}
  0 \leq{} 
  \la g_{t} + \ls H_{t} + (\beta_t + \lambda_t \|s_{t}\| )  I\rs s_{t},x-u_{t}\ra 
   ={} \la g_{t} + \ls H_{t} + \tau_t I\rs s_{t},x-u_{t}\ra, \qquad x \in \XCal.
\end{align*}
Thus, $\tau_t$ is the effective regularization coefficient associated with the computed step. Since $\tau_t$ depends on the unknown step length $\|s_t\|$, it can only be determined after solving the subproblem. Following~\citet{carmon2022optimal}, we scale down the growth of parameter $A_t$ if $\tau_t$ is too large. Specifically, we set $a_1 = \lambda_1/\tau_1$ and 
\begin{equation}
  \label{eq:A_relation_main}
  a_t= \min\lb\frac{\alpha_t A_{t-1}}{1-\alpha_t},
            \frac{\lambda_t}{\alpha_t\tau_t}\rb, \quad A_t = A_{t-1} + a_t,
\end{equation}
for $t\geq2$. Importantly, this choice ensures
\begin{equation}
   \label{eq:damping_condition}
  a_t \alpha_t \tau_t \leq \lambda_t.
\end{equation}
If the minimum in~\eqref{eq:A_relation_main}  is attained on the first term in~\eqref{eq:A_relation_main}, we perform the full update, resulting in $A_t=\frac{A_{t-1}}{1-\alpha_t}$ and $w_t = y_t$. Otherwise, the second term is active, so $a_t$ is reduced relative to full update and $w_t$ is a convex combination of $w_{t-1}$ and $y_t$. We call such an iteration damped.

\subsection{Convergence Analysis}

The following theorem characterizes its global convergence.
\begin{theorem}
  \label{thm:optimal_convergence}
  Let $f$ be convex and satisfy Assumption~\ref{as:lip_hess}. Let $g_t$ and $H_t$ in line~\ref{lin:deriv} satisfy Assumptions~\ref{as:stoch_grad} and~\ref{as:stoch_hess}, respectively. Let stochastic gradient $\widetilde{g}_t$ in line~\ref{lin:unbiased_grad} satisfy Assumptions~\ref{as:stoch_grad},~\ref{as:unbiased_grad}. Then Algorithm~\ref{alg:optimal_method} with $M\geq L_2$ and parameters~\eqref{eq:acc_params} satisfies, for every $T\geq1$,
  \begin{equation*}
    \E\lp f(w_T)-f^*\rp
    =\cO\ls\frac{\sigma_1D}{\sqrt T}
    +\frac{\sigma_2D^2}{T^2}+\frac{MD^3}{T^{7/2}}\rs.
  \end{equation*}
\end{theorem}

\textbf{Proof sketch.} We use the potential
\begin{equation*}
  E_t = A_t(f(w_t)-f^*)+\frac{\lambda_t}{2}\norm{z_t-x^*}^2.
\end{equation*}
Taking expectation and using H\"older inequality gives
\begin{equation*}
  \E \lp f(w_T) - f^\ast \rp \leq \E \lp E_T A_T^{-1} \rp \leq \ls \E E_T^2 \rs^{1/2} \ls \E A_T^{-2} \rs^{1/2}. 
\end{equation*}
The damping rule makes $A_T$ depend on the stochastic
derivatives. In the proof, we show that the following bounds hold
\begin{equation*}
  \ls\E E_T^2\rs^{1/2}\lesssim\lambda_TD^2,
  \qquad
  \ls\E A_T^{-2}\rs^{1/2}\lesssim\tfrac{D}{T^{7/2}}.
\end{equation*}
They give the rate of Theorem~\ref{thm:optimal_convergence}.

Using the damping condition~\eqref{eq:damping_condition},
we combine the optimality conditions of the two subproblems
in lines~\ref{lin:cubic_sub_acc} and~\ref{lin:avg_sub_acc}
with convexity to obtain
\begin{equation}
  \label{eq:potential_one_step_main}
  E_t+\tfrac{a_t\tau_t}{8\alpha_t}\norm{s_t}^2
  \leq E_{t-1}+\tfrac{D^2}{2}(\lambda_t-\lambda_{t-1})+\delta_t+\ell_t,
\end{equation} 
where $s_t=y_t-v_t$. 
The last two terms on the right are due to the stochastic oracle errors. As in the analysis of the basic method, for the $\widetilde{g}$'s error we use the decomposition
\begin{equation*}
    z_t-x^*=(z_t-z_{t-1})+(z_{t-1}-x^*).
\end{equation*}
The term involving $z_t-z_{t-1}$ is bounded using Young's inequality
and the diameter bound, and included in $\delta_t$, together with
the errors of $g_t$ and $H_t$. The remaining term
\begin{equation*}
    \ell_t
    =
    -a_t\la
    \widetilde g_t-\g(y_t),
    z_{t-1}-x^*
    \ra
\end{equation*}
is kept linear. Since $a_t$, $y_t$, and $z_{t-1}$ are fixed before the gradient $\widetilde g_t$ is queried, unbiasedness gives $\E\ell_t=0$.
Summing~\eqref{eq:potential_one_step_main}, we get
\begin{equation}
  \label{eq:potential_telescope_main}
  E_T + \frac{1}{8} \tsum_{t=1}^T \frac{a_t \tau_t}{\alpha_t}\|s_t\|^2 \leq \frac{\lambda_T D^2}{2} + \tsum_{t=1}^T \delta_t + \tsum_{t=1}^T \ell_t.
\end{equation}
Thus, to bound $\ls \E E_T^2 \rs^{1/2}$, we need to estimate second moments of the sums of $\delta$'s and $\ell$'s. Since $\ell_t$ has zero mean conditionally on the past, the cross terms vanish and
\begin{equation*}
  \E \ls \tsum_{t=1}^T \ell_t \rs^2 = \tsum_{t=1}^T \E \ell_t^2 \leq \tsum_{t=1}^T  D^2 \sigma_1^2 \E a_t^2\lesssim \sigma_1^2 T^6 \stackrel{\eqref{eq:acc_params}}{\lesssim} \lambda_T^2 D^4,
\end{equation*}
where $\tau_t \geq \beta_t$ and~\eqref{eq:damping_condition} give
$a_t\leq t^{3/2}/(\alpha_tD)$. The contraction bound $\norm{s_t}\leq\alpha_tD$
and the regularization parameters~\eqref{eq:acc_params}
also yield a second-moment bound of order
$\lambda_T^2D^4$ for $\tsum_{t=1}^T\delta_t$ (see Lemma~\ref{lem:optimal_error_moments}). Consequently, 
\begin{equation}
  \label{eq:potential_rate}
  \ls \E  E_T^2 \rs ^{1/2} \lesssim \lambda_T D^2.
\end{equation} 
It remains to show the bound on $\ls \E A_T^{-2} \rs^{1/2}$.  By~\eqref{eq:A_relation_main}, there are two types of iterations. If all iterations $t=2,\ldots,T$ are undamped, then
\begin{equation*}
  A_T^{-1} = A_1^{-1}\tprod_{t=2}^T (1-\alpha_t) \stackrel{\eqref{eq:acc_params}}{\lesssim} \frac{D}{T^4}.
\end{equation*}
On damped iterations, $a_t=\lambda_t/(\alpha_t\tau_t)$, so
\begin{equation*}
  A_t^{-1}\leq(1-\alpha_t)A_{t-1}^{-1}+\alpha_t^3\ls\tfrac{D}{t^{3/2}}+\norm{s_t}\rs,
\end{equation*}
while at undamped iterations the last term is absent. Multiplying by $t(t+1)(t+2)(t+3)$ and summing gives
\begin{equation*}
  A_T^{-1}\lesssim\frac{D}{T^{7/2}}
  +\frac{1}{T^4}\sum_{\substack{t\leq T\\a_t<\alpha_tA_t}}\frac{\norm{s_t}}{\alpha_t}.
\end{equation*}
On damped iterations,~\eqref{eq:damping_condition}
holds with equality, so the squared-step term
in~\eqref{eq:potential_one_step_main} is
$\lambda_t\norm{s_t}^2/(8\alpha_t^2)$.
A weighted summation of the normalized one-step bound
controls the second moment of the sum above by
$\cO(D^2T)$  (see Lemma~\ref{lem:optimal_movement}).
Substituting this estimate and using Minkowski inequality yields
\begin{equation*}
  \ls\E A_T^{-2}\rs^{1/2}
  \lesssim\frac{D}{T^{7/2}}.
\end{equation*}
Combining this bound with~\eqref{eq:potential_rate} and using the definition of $\lambda_t$ yields the desired rate. 

The full proof is given in Appendix~\ref{sec:proof_optimal}.

\section{Batch sizes and parallel efficiency}

Let us compare this rate with the optimal stochastic first-order method AC-SA~\cite{lan2012optimal}, which achieves
\begin{equation*}
    \E\big[f(w_T)-f^*\big] = \cO\Big(\tfrac{\sigma_1 D}{\sqrt{T}} + \tfrac{L_1 D^2}{T^2}\Big).
\end{equation*}
 In this section, for simplicity, we assume that the stochastic gradient term dominates other terms in all complexity bounds. The stochastic term $\frac{\sigma_1 D}{\sqrt{T}}$ is the same in both bounds.  Hence, both methods need the same total number of stochastic gradients, $N = \cO\big(\frac{\sigma_1^2 D^2}{\varepsilon^2}\big)$, to reach accuracy $\varepsilon$. This is the optimal stochastic gradient complexity. However, the two methods differ in how this budget can be split between the number of iterations $T$ and the batch size $b$.

With mini-batches of size $b$, the stochastic term becomes $\frac{\sigma_1 D}{\sqrt{bT}}$. So the total number of stochastic gradients $N = bT$ stays the same, while the number of iterations is determined by the deterministic terms. For AC-SA, this number is $T = \cO\big((L_1 D^2/\varepsilon)^{1/2}\big)$. For our method, it is
$T = \Big(\frac{M D^3}{\varepsilon}\Big)^{2/7},$
when the Hessian estimates are accurate enough, namely, when
\begin{equation}
    \label{eq:sigma2_requirement}
    \sigma_2 \leq \Big(\tfrac{M^4 \varepsilon^3}{D^2}\Big)^{1/7}.
\end{equation}
In this case, our method needs only $\cO(\varepsilon^{-2/7})$ iterations instead of $\cO(\varepsilon^{-1/2})$ for AC-SA. As a result, we can use a larger gradient batch size
\begin{equation*}
    b = \cO\Big(\tfrac{\sigma_1^2 D^2}{\varepsilon^2 T}\Big) = \cO\Big(\tfrac{\sigma_1^2 D^{8/7}}{M^{2/7}\varepsilon^{12/7}}\Big),
\end{equation*}
while keeping the same total number of stochastic gradients.

Condition \eqref{eq:sigma2_requirement} can be satisfied by Hessian mini-batching. Assume that the stochastic Hessians are unbiased and satisfy $\E\norm{H(x,\xi) - \nabla^2 f(x)}_F^2 \leq \sigma_H^2$. Let $\bar H$ be the average of $b_H$ independent samples. Then $\E\norm{\bar H - \nabla^2 f(x)}^2 \leq \sigma_H^2/b_H$, so $\sigma_2 \leq \sigma_H/\sqrt{b_H}$. Therefore, \eqref{eq:sigma2_requirement} holds for the Hessian batch size
\begin{equation*}
    b_H = \cO\Big(\tfrac{\sigma_H^2 D^{4/7}}{M^{8/7}\varepsilon^{6/7}}\Big).
\end{equation*}
This Hessian batch size grows as $\cO(\varepsilon^{-6/7})$, which is much slower than the gradient batch size, which grows as $\cO(\varepsilon^{-12/7})$. The total number of stochastic Hessians is $\cO(\varepsilon^{-8/7})$, which is much smaller than the total number of stochastic gradients, $\cO(\varepsilon^{-2})$.

Let us now discuss the parallel setting. In large clusters, the main bottleneck is usually communication/iterations, not computation. Therefore, we assume that the cost of one communication round dominates all additional computational costs, such as solving the subproblem. Under condition \eqref{eq:sigma2_requirement}, our method needs $T = \cO\big((MD^3/\varepsilon)^{2/7}\big)$ iterations, while AC-SA needs $T = \cO\big((L_1 D^2/\varepsilon)^{1/2}\big)$ iterations with the same total number of stochastic gradients. Therefore, our method requires fewer communication rounds, and it is faster than first-order methods in this setting. This observation opens the direction for a detailed distributed version of our method, which would also take into account the communication cost of the Hessian information, for example, by using $\beta$-similarity, Hessian-vector products, compression, or sketching.

\clearpage
\begin{figure}[!t]
  \centering
  \begin{minipage}{0.58\textwidth}
    \centering
    \includegraphics[width=\linewidth]{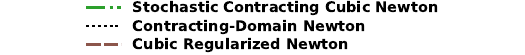}
    \includegraphics[width=\linewidth]{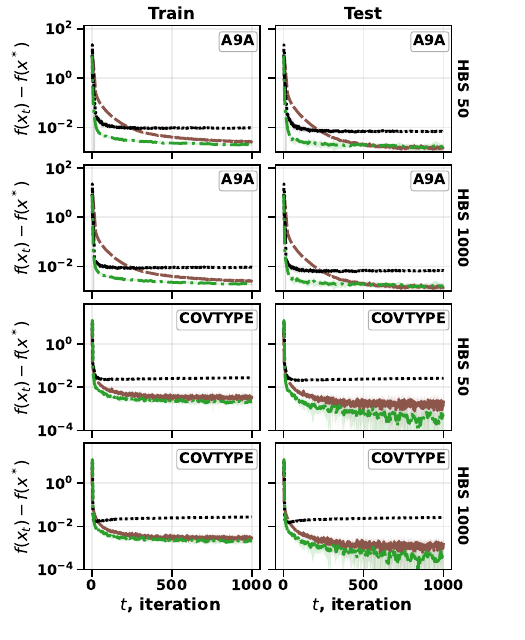}
  \end{minipage}
  \caption{Non-accelerated methods.}
  \label{fig:ball_reruns}
\end{figure}

\section{Experiments}
\label{sec:experiments}

We evaluate the methods on logistic regression using the \texttt{a9a} and \texttt{covtype} datasets from LIBSVM~\citep{chang2011libsvm}. 
We use a fixed random 80/20 train/test split. All coordinates of the initial point $x_0$ are set to~3. We optimize over the Euclidean ball $\XCal$ with center $x_0$ and diameter $D=2\norm{x_0}$.
The source code is \href{https://anonymous.4open.science/r/accelerated_stoch-4832/README.md}{available here}.

In Figure~\ref{fig:ball_reruns}, CDN is parameter-free. For CRN and SCCN, we use the same tuning budget of one scalar parameter: $M$ for CRN and $\sigma_1/D$ for SCCN, with $M=0$ for SCCN. Consequently, both
CRN and SCCN have exactly one tuned scalar parameter. 

\textbf{Discussion.} 
In Figure~\ref{fig:ball_reruns}, Algorithm~\ref{alg:basic_method} outperforms both CRN and CDN, showing that our additional quadratic regularization benefits convergence in the stochastic setting. CDN, which does not use our additional regularization, initially decreases the objective but then stagnates at
a higher error level. CRN continues to converge, but its progress is slower.

\clearpage
\subsection*{Limitations}
Our analysis is restricted to convex optimization over bounded convex
sets, and the convergence bounds depend on the domain diameter $D$. Developing an optimal stochastic second-order method for unconstrained convex optimization over $\R^d$ remains an open problem. Our methods require knowledge of problem parameters. Developing adaptive or parameter-free variants is an interesting direction for future work.

\subsection*{AI use statement}

In this work, we used generative AI tools to polish writing, to assist with coding, and to assist with parts of the theoretical analysis.
We have not used generative AI tools for retrieval and discovery (e.g., finding related work), writing sections of the paper, or generating synthetic datasets.
All research ideas, including the proposed methods, as well as the core analysis, are original work of the authors.
We have reviewed all AI-assisted work.

\textit{Text and code.}
We used OpenAI ChatGPT, Anthropic Claude, Codex, and Claude Code to edit author-written text and to assist with code development.

\textit{Mathematical claims.}
The proposed methods and the main ideas of the analysis are our own. We used ChatGPT, Codex, and Claude Code to transcribe our handwritten proofs into \LaTeX, to review the proofs, and to assist with their iterative refinement. 

We take responsibility for the final content of this work, including text, claims, or artifacts produced with the aid of generative AI.

\bibliography{agafonov}

\begin{thebibliography}{40}
\providecommand{\natexlab}[1]{#1}
\providecommand{\url}[1]{\texttt{#1}}
\expandafter\ifx\csname urlstyle\endcsname\relax
  \providecommand{\doi}[1]{doi: #1}\else
  \providecommand{\doi}{doi: \begingroup \urlstyle{rm}\Url}\fi

\bibitem[Agafonov et~al.(2021)Agafonov, Dvurechensky, Scutari, Gasnikov, Kamzolov, Lukashevich, and Daneshmand]{agafonov2021accelerated}
Artem Agafonov, Pavel Dvurechensky, Gesualdo Scutari, Alexander Gasnikov, Dmitry Kamzolov, Aleksandr Lukashevich, and Amir Daneshmand.
\newblock An accelerated second-order method for distributed stochastic optimization.
\newblock In \emph{2021 60th IEEE Conference on Decision and Control (CDC)}, pages 2407--2413, 2021.
\newblock ISBN 2576-2370.
\newblock \doi{10.1109/CDC45484.2021.9683400}.
\newblock URL \url{https://doi.org/10.1109/CDC45484.2021.9683400}.

\bibitem[Agafonov et~al.(2024{\natexlab{a}})Agafonov, Kamzolov, Dvurechensky, Gasnikov, and Tak{\'a}{\v{c}}]{agafonov2023inexact}
Artem Agafonov, Dmitry Kamzolov, Pavel Dvurechensky, Alexander Gasnikov, and Martin Tak{\'a}{\v{c}}.
\newblock Inexact tensor methods and their application to stochastic convex optimization.
\newblock \emph{Optimization Methods and Software}, 39\penalty0 (1):\penalty0 42--83, 2024{\natexlab{a}}.
\newblock \doi{10.1080/10556788.2023.2261604}.
\newblock URL \url{https://doi.org/10.1080/10556788.2023.2261604}.

\bibitem[Agafonov et~al.(2024{\natexlab{b}})Agafonov, Kamzolov, Gasnikov, Kavis, Antonakopoulos, Cevher, and Tak{\'a}{\v{c}}]{agafonov2024advancing}
Artem Agafonov, Dmitry Kamzolov, Alexander Gasnikov, Ali Kavis, Kimon Antonakopoulos, Volkan Cevher, and Martin Tak{\'a}{\v{c}}.
\newblock Advancing the lower bounds: an accelerated, stochastic, second-order method with optimal adaptation to inexactness.
\newblock In \emph{The Twelfth International Conference on Learning Representations}, 2024{\natexlab{b}}.
\newblock URL \url{https://openreview.net/forum?id=otU31x3fus}.

\bibitem[Agafonov et~al.(2026{\natexlab{a}})Agafonov, Ashabokov, D'yakonov, Tak{\'a}{\v{c}}, Gasnikov, and Kamzolov]{agafonov2026optimal}
Artem Agafonov, Aslan Ashabokov, Alexander D'yakonov, Martin Tak{\'a}{\v{c}}, Alexander Gasnikov, and Dmitry Kamzolov.
\newblock Optimal parameter-free second-order acceleration with inexact hessians under generalized smoothness.
\newblock \emph{arXiv preprint arXiv:2609.32977}, 2026{\natexlab{a}}.

\bibitem[Agafonov et~al.(2026{\natexlab{b}})Agafonov, Novitskii, Rogozin, Sokolov, Kamzolov, Dyakonov, Tak’a{\v{c}}, and Gasnikov]{agafonov2026decentralized}
Artem Agafonov, Anton Novitskii, Alexander Rogozin, Yury Sokolov, Dmitry Kamzolov, Alexander~G Dyakonov, Martin Tak’a{\v{c}}, and Alexander Gasnikov.
\newblock Decentralized inexact cubic newton method with consensus procedure.
\newblock \emph{IEEE Access}, 2026{\natexlab{b}}.

\bibitem[Agarwal and Hazan(2018)]{pmlr-v75-agarwal18a}
Naman Agarwal and Elad Hazan.
\newblock Lower bounds for higher-order convex optimization.
\newblock In Sébastien Bubeck, Vianney Perchet, and Philippe Rigollet, editors, \emph{Proceedings of the 31st Conference On Learning Theory}, volume~75 of \emph{Proceedings of Machine Learning Research}, pages 774--792. PMLR, 06--09 Jul 2018.
\newblock URL \url{https://proceedings.mlr.press/v75/agarwal18a.html}.

\bibitem[Antonakopoulos et~al.(2022)Antonakopoulos, Kavis, and Cevher]{antonakopoulos2022extra}
Kimon Antonakopoulos, Ali Kavis, and Volkan Cevher.
\newblock Extra-{N}ewton: A first approach to noise-adaptive accelerated second-order methods.
\newblock In S~Koyejo, S~Mohamed, A~Agarwal, D~Belgrave, K~Cho, and A~Oh, editors, \emph{Advances in Neural Information Processing Systems}, volume~35, pages 29859--29872. Curran Associates, Inc., 2022.
\newblock URL \url{https://proceedings.neurips.cc/paper_files/paper/2022/file/c10804702be5a0cca89331315413f1a2-Paper-Conference.pdf}.

\bibitem[Arjevani et~al.(2019)Arjevani, Shamir, and Shiff]{arjevani2019oracle}
Yossi Arjevani, Ohad Shamir, and Ron Shiff.
\newblock Oracle complexity of second-order methods for smooth convex optimization.
\newblock \emph{Mathematical Programming}, 178:\penalty0 327--360, 2019.
\newblock ISSN 1436-4646.
\newblock \doi{10.1007/s10107-018-1293-1}.
\newblock URL \url{https://doi.org/10.1007/s10107-018-1293-1}.

\bibitem[Bennett(1916)]{bennett1916newton}
Albert~A Bennett.
\newblock Newton's method in general analysis.
\newblock \emph{Proceedings of the National Academy of Sciences}, 2\penalty0 (10):\penalty0 592--598, 1916.

\bibitem[Carmon et~al.(2022)Carmon, Hausler, Jambulapati, Jin, and Sidford]{carmon2022optimal}
Yair Carmon, Danielle Hausler, Arun Jambulapati, Yujia Jin, and Aaron Sidford.
\newblock Optimal and adaptive {M}onteiro-{S}vaiter acceleration.
\newblock In S~Koyejo, S~Mohamed, A~Agarwal, D~Belgrave, K~Cho, and A~Oh, editors, \emph{Advances in Neural Information Processing Systems}, volume~35, pages 20338--20350. Curran Associates, Inc., 2022.
\newblock URL \url{https://proceedings.neurips.cc/paper_files/paper/2022/file/7ff97417474268e6b5a38bcbfae04944-Paper-Conference.pdf}.

\bibitem[Chang and Lin(2011)]{chang2011libsvm}
Chih-Chung Chang and Chih-Jen Lin.
\newblock {LIBSVM}: A library for support vector machines.
\newblock \emph{ACM Trans. Intell. Syst. Technol.}, 2\penalty0 (3), May 2011.
\newblock ISSN 2157-6904.
\newblock \doi{10.1145/1961189.1961199}.
\newblock URL \url{https://doi.org/10.1145/1961189.1961199}.

\bibitem[Chayti et~al.(2025)Chayti, Doikov, and Jaggi]{chayti2025improving}
El~Mahdi Chayti, Nikita Doikov, and Martin Jaggi.
\newblock Improving stochastic cubic {N}ewton with momentum.
\newblock In \emph{Proceedings of the 28th International Conference on Artificial Intelligence and Statistics}, volume 258 of \emph{Proceedings of Machine Learning Research}. PMLR, 2025.

\bibitem[Chen et~al.(2026)Chen, Liu, Luo, Lui, and Zhang]{chen2026optimal}
Lesi Chen, Chengchang Liu, Luo Luo, John Lui, and Jingzhao Zhang.
\newblock Optimal convex optimization with inexact second-order oracles.
\newblock \emph{arXiv preprint arXiv:2607.24520}, 2026.

\bibitem[Defazio and Gower(2021)]{defazio2021power}
Aaron Defazio and Robert~M Gower.
\newblock The power of factorial powers: New parameter settings for (stochastic) optimization.
\newblock In \emph{Asian Conference on Machine Learning}, pages 49--64. PMLR, 2021.

\bibitem[Doikov(2026)]{doikov2026minimizing}
Nikita Doikov.
\newblock Minimizing quasi-self-concordant functions by gradient regularization of newton method.
\newblock \emph{Mathematical Programming}, 219\penalty0 (1):\penalty0 583--621, 2026.

\bibitem[Doikov and Nesterov(2020)]{doikov2020convex}
Nikita Doikov and Yurii Nesterov.
\newblock Convex optimization based on global lower second-order models.
\newblock In H.~Larochelle, M.~Ranzato, R.~Hadsell, M.F. Balcan, and H.~Lin, editors, \emph{Advances in Neural Information Processing Systems}, volume~33, pages 16546--16556. Curran Associates, Inc., 2020.
\newblock URL \url{https://proceedings.neurips.cc/paper_files/paper/2020/file/c0c3a9fb8385d8e03a46adadde9af3bf-Paper.pdf}.

\bibitem[Doikov and Nesterov(2023)]{doikov2023gradient}
Nikita Doikov and Yurii Nesterov.
\newblock Gradient regularization of {N}ewton method with {B}regman distances.
\newblock \emph{Mathematical Programming}, 2023.
\newblock ISSN 1436-4646.
\newblock \doi{10.1007/s10107-023-01943-7}.
\newblock URL \url{https://doi.org/10.1007/s10107-023-01943-7}.

\bibitem[Doikov et~al.(2023)Doikov, Chayti, and Jaggi]{doikov2023second}
Nikita Doikov, El~Mahdi Chayti, and Martin Jaggi.
\newblock Second-order optimization with lazy {H}essians.
\newblock In Andreas Krause, Emma Brunskill, Kyunghyun Cho, Barbara Engelhardt, Sivan Sabato, and Jonathan Scarlett, editors, \emph{Proceedings of the 40th International Conference on Machine Learning}, volume 202, pages 8138--8161. PMLR, 9 2023.
\newblock URL \url{https://proceedings.mlr.press/v202/doikov23a.html}.

\bibitem[Dvurechensky et~al.(2022)Dvurechensky, Kamzolov, Lukashevich, Lee, Ordentlich, Uribe, and Gasnikov]{dvurechensky2022hyperfast}
Pavel Dvurechensky, Dmitry Kamzolov, Aleksandr Lukashevich, Soomin Lee, Erik Ordentlich, César~A Uribe, and Alexander Gasnikov.
\newblock Hyperfast second-order local solvers for efficient statistically preconditioned distributed optimization.
\newblock \emph{EURO Journal on Computational Optimization}, 10:\penalty0 100045, 2022.
\newblock ISSN 2192-4406.
\newblock \doi{https://doi.org/10.1016/j.ejco.2022.100045}.
\newblock URL \url{https://www.sciencedirect.com/science/article/pii/S2192440622000211}.

\bibitem[Frank and Wolfe(1956)]{frank1956algorithm}
Marguerite Frank and Philip Wolfe.
\newblock An algorithm for quadratic programming.
\newblock \emph{Naval research logistics quarterly}, 3\penalty0 (1-2):\penalty0 95--110, 1956.

\bibitem[Ghadimi et~al.(2017)Ghadimi, Liu, and Zhang]{ghadimi2017second}
Saeed Ghadimi, Han Liu, and Tong Zhang.
\newblock Second-order methods with cubic regularization under inexact information.
\newblock \emph{arXiv preprint arXiv:1710.05782}, 2017.

\bibitem[Hanzely et~al.(2022)Hanzely, Kamzolov, Pasechnyuk, Gasnikov, Richt{\'a}rik, and Tak{\'a}{\v{c}}]{hanzely2022damped}
Slavom\'{\i}r Hanzely, Dmitry Kamzolov, Dmitry Pasechnyuk, Alexander Gasnikov, Peter Richt{\'a}rik, and Martin Tak{\'a}{\v{c}}.
\newblock A damped {N}ewton method achieves global $\mathcal{O}\left (\frac{1}{k^{2}}\right) $ and local quadratic convergence rate.
\newblock In S.~Koyejo, S.~Mohamed, A.~Agarwal, D.~Belgrave, K.~Cho, and A.~Oh, editors, \emph{Advances in Neural Information Processing Systems}, volume~35, pages 25320--25334. Curran Associates, Inc., 2022.
\newblock URL \url{https://proceedings.neurips.cc/paper_files/paper/2022/file/a1f0c0cd6caaa4863af5f12608edf63e-Paper-Conference.pdf}.

\bibitem[Kamzolov et~al.(2025{\natexlab{a}})Kamzolov, Agafonov, Pasechnyuk, Gasnikov, and Tak{\'a}{\v{c}}]{kamzolov2025optami}
Dmitry Kamzolov, Artem Agafonov, Dmitry Pasechnyuk, Alexander Gasnikov, and Martin Tak{\'a}{\v{c}}.
\newblock {OPTAMI}: Global superlinear convergence of high-order methods.
\newblock In \emph{The Thirteenth International Conference on Learning Representations}, 2025{\natexlab{a}}.
\newblock URL \url{https://openreview.net/forum?id=Cpr6Wv2tfr}.

\bibitem[Kamzolov et~al.(2025{\natexlab{b}})Kamzolov, Ziu, Agafonov, and Tak{\'a}{\v{c}}]{kamzolov2026accelerated}
Dmitry Kamzolov, Klea Ziu, Artem Agafonov, and Martin Tak{\'a}{\v{c}}.
\newblock Accelerated adaptive cubic regularized quasi-{N}ewton methods.
\newblock \emph{Journal of Optimization Theory and Applications}, 208:\penalty0 27, 2025{\natexlab{b}}.
\newblock ISSN 1573-2878.
\newblock \doi{10.1007/s10957-025-02804-3}.
\newblock URL \url{https://doi.org/10.1007/s10957-025-02804-3}.

\bibitem[Kantorovich(1949)]{kantorovich1949newton}
Leonid~Vitalyevich Kantorovich.
\newblock On {N}ewton's method.
\newblock \emph{Trudy Matematicheskogo Instituta imeni VA Steklova}, 28:\penalty0 104--144, 1949.
\newblock (In Russian).

\bibitem[Kovalev and Gasnikov(2022)]{kovalev2022first}
Dmitry Kovalev and Alexander Gasnikov.
\newblock The first optimal acceleration of high-order methods in smooth convex optimization.
\newblock In S~Koyejo, S~Mohamed, A~Agarwal, D~Belgrave, K~Cho, and A~Oh, editors, \emph{Advances in Neural Information Processing Systems}, volume~35, pages 35339--35351. Curran Associates, Inc., 2022.
\newblock URL \url{https://proceedings.neurips.cc/paper_files/paper/2022/file/e56f394bbd4f0ec81393d767caa5a31b-Paper-Conference.pdf}.

\bibitem[Lan(2012)]{lan2012optimal}
Guanghui Lan.
\newblock An optimal method for stochastic composite optimization.
\newblock \emph{Mathematical Programming}, 133:\penalty0 365--397, 2012.
\newblock ISSN 1436-4646.
\newblock \doi{10.1007/s10107-010-0434-y}.
\newblock URL \url{https://doi.org/10.1007/s10107-010-0434-y}.

\bibitem[Mishchenko(2023)]{mishchenko2023regularized}
Konstantin Mishchenko.
\newblock Regularized {N}ewton method with global $\mathcal{O}\left (\frac{1}{k^{2}}\right)$ convergence.
\newblock \emph{SIAM Journal on Optimization}, 33:\penalty0 1440--1462, 2023.
\newblock \doi{10.1137/22M1488752}.
\newblock URL \url{https://doi.org/10.1137/22M1488752}.

\bibitem[Mokhtari et~al.(2020)Mokhtari, Hassani, and Karbasi]{mokhtari2020stochastic}
Aryan Mokhtari, Hamed Hassani, and Amin Karbasi.
\newblock Stochastic conditional gradient methods: From convex minimization to submodular maximization.
\newblock \emph{Journal of machine learning research}, 21\penalty0 (105):\penalty0 1--49, 2020.

\bibitem[Monteiro and Svaiter(2013)]{monteiro2013accelerated}
Renato D~C Monteiro and B~F Svaiter.
\newblock An accelerated hybrid proximal extragradient method for convex optimization and its implications to second-order methods.
\newblock \emph{SIAM Journal on Optimization}, 23:\penalty0 1092--1125, 2013.
\newblock \doi{10.1137/110833786}.
\newblock URL \url{https://doi.org/10.1137/110833786}.

\bibitem[Nemirovski et~al.(2009)Nemirovski, Juditsky, Lan, and Shapiro]{nemirovski2009robust}
A~Nemirovski, A~Juditsky, G~Lan, and A~Shapiro.
\newblock Robust stochastic approximation approach to stochastic programming.
\newblock \emph{SIAM Journal on Optimization}, 19:\penalty0 1574--1609, 2009.
\newblock \doi{10.1137/070704277}.
\newblock URL \url{https://doi.org/10.1137/070704277}.

\bibitem[Nemirovski and Yudin(1983)]{nemirovskiĭ1983problem}
A.S. Nemirovski and D.B. Yudin.
\newblock \emph{Problem Complexity and Method Efficiency in Optimization}.
\newblock A Wiley-Interscience publication. Wiley, 1983.
\newblock ISBN 9780471103455.
\newblock URL \url{https://books.google.fr/books?id=6ULvAAAAMAAJ}.

\bibitem[Nesterov(1983)]{nesterov1983method}
Yurii Nesterov.
\newblock A method for solving the convex programming problem with convergence rate $\mathcal{O}\left (\frac{1}{k^{2}}\right) $.
\newblock \emph{Doklady Akademii Nauk SSSR}, 269\penalty0 (3):\penalty0 543--547, 1983.
\newblock (In Russian).

\bibitem[Nesterov(2008)]{nesterov2008accelerating}
Yurii Nesterov.
\newblock Accelerating the cubic regularization of {N}ewton’s method on convex problems.
\newblock \emph{Mathematical Programming}, 112:\penalty0 159--181, 2008.
\newblock ISSN 1436-4646.
\newblock \doi{10.1007/s10107-006-0089-x}.
\newblock URL \url{https://doi.org/10.1007/s10107-006-0089-x}.

\bibitem[Nesterov(2018)]{nesterov2018lectures}
Yurii Nesterov.
\newblock \emph{Lectures on Convex Optimization}.
\newblock Springer Cham, 2 edition, 2018.
\newblock ISBN 978-3-319-91577-7.
\newblock \doi{10.1007/978-3-319-91578-4}.

\bibitem[Nesterov and Polyak(2006)]{nesterov2006cubic}
Yurii Nesterov and Boris~T Polyak.
\newblock Cubic regularization of {N}ewton method and its global performance.
\newblock \emph{Mathematical Programming}, 108:\penalty0 177--205, 2006.
\newblock \doi{10.1007/s10107-006-0706-8}.
\newblock URL \url{https://doi.org/10.1007/s10107-006-0706-8}.

\bibitem[Pedregosa et~al.(2011)Pedregosa, Varoquaux, Gramfort, Michel, Thirion, Grisel, Blondel, Prettenhofer, Weiss, Dubourg, et~al.]{pedregosa2011scikit}
Fabian Pedregosa, Ga{\"e}l Varoquaux, Alexandre Gramfort, Vincent Michel, Bertrand Thirion, Olivier Grisel, Mathieu Blondel, Peter Prettenhofer, Ron Weiss, Vincent Dubourg, et~al.
\newblock Scikit-learn: Machine learning in python.
\newblock \emph{the Journal of machine Learning research}, 12:\penalty0 2825--2830, 2011.

\bibitem[Semenov et~al.(2026)Semenov, Jaggi, and Doikov]{semenov2026gradient}
Andrei Semenov, Martin Jaggi, and Nikita Doikov.
\newblock Gradient-normalized smoothness for optimization with approximate {H}essians.
\newblock In C.~Vondrick, B.~Hariharan, C.~Raffel, L.~Pinto, D.~Yang, and A.~Faust, editors, \emph{International Conference on Learning Representations}, volume 2026, pages 13974--14029, 2026.
\newblock URL \url{https://proceedings.iclr.cc/paper_files/paper/2026/file/1759a83b007f0685c3fbc460fa1b6395-Paper-Conference.pdf}.

\bibitem[Zhang et~al.(2026)Zhang, Chen, Liu, and Zhang]{zhang2026nearoptimalconvexoptimizationlazy}
Xinliang Zhang, Lesi Chen, Chengchang Liu, and Jingzhao Zhang.
\newblock Near-optimal convex optimization with lazy second-order oracles, 2026.
\newblock URL \url{https://arxiv.org/abs/2610.03222}.

\bibitem[Zhang and Lin(2015)]{zhang2015disco}
Yuchen Zhang and Xiao Lin.
\newblock {D}i{SCO}: Distributed optimization for self-concordant empirical loss.
\newblock In Francis Bach and David Blei, editors, \emph{Proceedings of the 32nd International Conference on Machine Learning}, volume~37 of \emph{Proceedings of Machine Learning Research}, pages 362--370, Lille, France, 07--09 Jul 2015. PMLR.
\newblock URL \url{https://proceedings.mlr.press/v37/zhangb15.html}.

\end{thebibliography}
\bibliographystyle{plainnat}

\clearpage
\appendix

\section{Proof of Basic Method}
\label{sec:proof_basic}

Let us restate Lemma~\ref{lem:basic_one_step_main}.

\begin{lemma}
  \label{lem:basic_one_step}
  Let $f$ be convex and satisfy Assumption~\ref{as:lip_hess}, and let $M\geq 0$. For the iterates of Algorithm~\ref{alg:basic_method}, with $0<\alpha_t\leq1$, $\eta_t\geq0$, and $H_t\succeq0$, we have
  \begin{equation}
    \label{eq:basic_one_step}
    \begin{aligned}
      f(x_{t+1})-f^*
      \leq{}& (1-\alpha_t)\ls f(x_t)-f^*\rs \\
      &+\frac{\alpha_t\eta_t}{2}\ls
        \norm{u_t-x^*}^2-\norm{u_{t+1}-x^*}^2-\norm{u_{t+1}-u_t}^2
      \rs \\
      &+\alpha_t\la\g(x_t)-g_t,u_{t+1}-x^*\ra
      +\alpha_t^2D^2\norm{H_t-\h(x_t)}
      +\frac{\max\lb M,L_2\rb}{3}\alpha_t^3D^3.
    \end{aligned}
  \end{equation}
\end{lemma}
\begin{proof}
  Optimality of $u_{t+1}$ over $\XCal$ gives
  \begin{equation*}
    \la\nabla\omega_t^M(u_{t+1}),x-u_{t+1}\ra\geq0,
    \qquad x\in\XCal.
  \end{equation*}
  Since $H_t\succeq0$, the model $\omega_t^M$ is convex and, when $\eta_t>0$, is $\eta_t$-strongly convex. Therefore,
  \begin{equation*}
    \omega_t^M(x)
    \geq\omega_t^M(u_{t+1})
      +\la\nabla\omega_t^M(u_{t+1}),x-u_{t+1}\ra
      +\frac{\eta_t}{2}\norm{x-u_{t+1}}^2
    \geq\omega_t^M(u_{t+1})+\frac{\eta_t}{2}\norm{x-u_{t+1}}^2.
  \end{equation*}
  Rearranging yields
  \begin{equation}
    \label{eq:basic_model_comparison}
    \omega_t^M(u_{t+1})
    \leq\omega_t^M(x)-\frac{\eta_t}{2}\norm{u_{t+1}-x}^2,
    \qquad x\in\XCal.
  \end{equation}

  By Lemma~\ref{lem:taylor_remainder} with $x=x_t$ and $y=x_{t+1}$, using $x_{t+1}-x_t=\alpha_t(u_{t+1}-x_t)$, we obtain
  \begin{equation}
    \label{eq:basic_model_upper}
    \begin{aligned}
      f(x_{t+1})
      &\leq f(x_t)+\alpha_t\la\g(x_t),u_{t+1}-x_t\ra \\
      &\quad+\frac{\alpha_t^2}{2}\la\h(x_t)(u_{t+1}-x_t),u_{t+1}-x_t\ra
      +\frac{L_2\alpha_t^3}{6}\norm{u_{t+1}-x_t}^3 \\
      &\stackrel{\eqref{eq:cubic_model}}{=} f(x_t)+\alpha_t\omega_t^M(u_{t+1})
      -\frac{\alpha_t\eta_t}{2}\norm{u_{t+1}-u_t}^2 
      +\alpha_t\la\g(x_t)-g_t,u_{t+1}-x_t\ra \\
      &\quad+\frac{\alpha_t^2}{2}\la\ls\h(x_t)-H_t\rs(u_{t+1}-x_t),u_{t+1}-x_t\ra 
      +\frac{(L_2-M)\alpha_t^3}{6}\norm{u_{t+1}-x_t}^3 \\
      &\leq f(x_t)+\alpha_t\omega_t^M(u_{t+1})
      -\frac{\alpha_t\eta_t}{2}\norm{u_{t+1}-u_t}^2 
      +\alpha_t\la\g(x_t)-g_t,u_{t+1}-x_t\ra \\
      &\quad+\frac{\alpha_t^2D^2}{2}\norm{H_t-\h(x_t)} +\frac{(L_2-M)_{+}}{6}\alpha_t^3D^3\\
      &\stackrel{\eqref{eq:basic_model_comparison}}{\leq} f(x_t)+\alpha_t\omega_t^M(x)
      -\frac{\alpha_t\eta_t}{2}\ls\norm{u_{t+1}-x}^2+\norm{u_{t+1}-u_t}^2\rs
      +\alpha_t\la\g(x_t)-g_t,u_{t+1}-x_t\ra \\
      &\quad +\frac{\alpha_t^2D^2}{2}\norm{H_t-\h(x_t)} +\frac{(L_2-M)_+}{6}\alpha_t^3D^3,
    \end{aligned}
  \end{equation}
  where the second inequality uses $\norm{u_{t+1}-x_t}\leq D$ and
  \begin{equation*}
    (L_2-M)\norm{u_{t+1}-x_t}^3\leq(L_2-M)_+D^3,
    \qquad (L_2-M)_+=\max\lb L_2-M,0\rb.
  \end{equation*}

  Next, $(1-\alpha_t)x_t+\alpha_tx\in\XCal$ since $0<\alpha_t\leq1$. By Lemma~\ref{lem:taylor_remainder} at this point
  \begin{equation}
    \begin{aligned}
      \label{eq:basic_comparator_taylor}
      f\ls(1-\alpha_t)x_t+\alpha_tx\rs 
      \geq{} & f(x_t)+\alpha_t\la\g(x_t),x-x_t\ra\\
      &+\frac{\alpha_t^2}{2}\la\h(x_t)(x-x_t),x-x_t\ra
      -\frac{L_2\alpha_t^3}{6}\norm{x-x_t}^3.
    \end{aligned}
  \end{equation}
  Therefore, by~\eqref{eq:basic_comparator_taylor} and the definition of $\omega_t^M(x)$~\eqref{eq:cubic_model},
  \begin{equation}
    \label{eq:basic_comparator_bound}
    \begin{aligned}
      f(x_t)+\alpha_t\omega_t^M(x)
      \stackrel{\eqref{eq:cubic_model}}{=}& f(x_t)+\alpha_t\la\g(x_t),x-x_t\ra
      +\frac{\alpha_t^2}{2}\la\h(x_t)(x-x_t),x-x_t\ra\\
      &+\alpha_t\la g_t-\g(x_t),x-x_t\ra\\
      &+\frac{\alpha_t^2}{2}\la\ls H_t-\h(x_t)\rs(x-x_t),x-x_t\ra\\
      &+\frac{M\alpha_t^3}{6}\norm{x-x_t}^3
      +\frac{\alpha_t\eta_t}{2}\norm{x-u_t}^2\\
      \stackrel{\eqref{eq:basic_comparator_taylor}}{\leq}& f\ls(1-\alpha_t)x_t+\alpha_tx\rs
      +\frac{\alpha_t\eta_t}{2}\norm{x-u_t}^2
      +\alpha_t\la g_t-\g(x_t),x-x_t\ra\\
      &+\frac{\alpha_t^2}{2}\la\ls H_t-\h(x_t)\rs(x-x_t),x-x_t\ra
      +\frac{M+L_2}{6}\alpha_t^3\norm{x-x_t}^3\\
      \leq{}& f\ls(1-\alpha_t)x_t+\alpha_tx\rs
      +\frac{\alpha_t\eta_t}{2}\norm{x-u_t}^2
      +\alpha_t\la g_t-\g(x_t),x-x_t\ra\\
      &+\frac{\alpha_t^2D^2}{2}\norm{H_t-\h(x_t)}
      +\frac{M+L_2}{6}\alpha_t^3D^3,
    \end{aligned}
  \end{equation}
  where the last inequality uses $\norm{x-x_t}\leq D$.
  Substituting~\eqref{eq:basic_comparator_bound} into~\eqref{eq:basic_model_upper} with $x=x^*$ and using the convexity
  \begin{equation*}
    f\ls(1-\alpha_t)x_t+\alpha_tx^*\rs
    \leq(1-\alpha_t)f(x_t)+\alpha_tf^*
  \end{equation*}
  together with $M+L_2+(L_2-M)_+=L_2+\max\lb M,L_2\rb\leq2\max\lb M,L_2\rb$ yields~\eqref{eq:basic_one_step}.
\end{proof}

Let us restate Theorem~\ref{thm:basic_convergence}. 

\begin{theorem}
  Let $f$ be convex. Let Assumptions~\ref{as:lip_hess},~\ref{as:stoch_grad},~\ref{as:unbiased_grad}, and~\ref{as:stoch_hess} hold. Then Algorithm~\ref{alg:basic_method} with $M\geq0$ and parameters~\eqref{eq:basic_reg} satisfies, for every integer $T\geq1$,
  \begin{equation}
    \label{eq:basic_rate}
      \E\lp f(x_T)-f^*\rp
      \leq\frac{9\max\lb M,L_2\rb D^3}{(T+1)(T+2)} 
      +\frac{9\sigma_2D^2}{2(T+2)}+\frac{9\sigma_1D}{4\sqrt T}.
  \end{equation}
  In particular, if $M\geq0$ and $M=\cO(L_2)$ with an absolute implied constant, then
  \begin{equation*}
    \E\lp f(x_T)-f^*\rp
    =\cO\ls\frac{\sigma_1D}{\sqrt T} + \frac{\sigma_2D^2}{T} + \frac{L_2D^3}{T^2}\rs.
  \end{equation*}
\end{theorem}

\begin{proof}
  Suppose first that $\sigma_1>0$, so that $\eta_t>0$. By~\eqref{eq:unbiased_grad} we have
  \begin{equation*}
    \E\lp\la\g(x_t)-g_t,u_t-x^*\ra
      \mid x_t,u_t\rp=0.
  \end{equation*}
  Splitting $u_{t+1}-x^*=(u_{t+1}-u_t)+(u_t-x^*)$ and applying Young's inequality gives
  \begin{gather*}
    \E\lp\la\g(x_t)-g_t,u_{t+1}-x^*\ra
      -\frac{\eta_t}{2}\norm{u_{t+1}-u_t}^2\rp 
    =\E\lp\la\g(x_t)-g_t,u_{t+1}-u_t\ra
      -\frac{\eta_t}{2}\norm{u_{t+1}-u_t}^2\rp \\
    \leq\frac{1}{2\eta_t}\E\norm{g_t-\g(x_t)}^2
    \stackrel{\eqref{eq:stoch_grad}}{\leq}\frac{\sigma_1^2}{2\eta_t}.
  \end{gather*}
  Also, by~\eqref{eq:stoch_hess},
  \begin{equation*}
    \E\norm{H_t-\h(x_t)}
    \leq\ls\E\norm{H_t-\h(x_t)}^2\rs^{1/2}
    \leq\sigma_2.
  \end{equation*}
  Taking expectations in~\eqref{eq:basic_one_step} and substituting these bounds yields
  \begin{equation}
    \label{eq:basic_expected_step}
    \begin{aligned}
      \E\lp f(x_{t+1})-f^*\rp
      \leq{}&(1-\alpha_t)\E\lp f(x_t)-f^*\rp 
      +\frac{\alpha_t\eta_t}{2}\ls
        \E\norm{u_t-x^*}^2-\E\norm{u_{t+1}-x^*}^2
      \rs  \\
      &+\frac{\max\lb M,L_2\rb}{3}\alpha_t^3D^3
      +\alpha_t^2\sigma_2D^2+\frac{\alpha_t\sigma_1^2}{2\eta_t}.
    \end{aligned}
  \end{equation}
  Multiplying~\eqref{eq:basic_expected_step} by $(t+1)(t+2)(t+3)$ and substituting $\alpha_t=3/(t+3)$ gives
  \begin{gather*}
      (t+1)(t+2)(t+3)\E\lp f(x_{t+1})-f^*\rp
      -t(t+1)(t+2)\E\lp f(x_t)-f^*\rp \\
      \leq\frac{3}{2}(t+1)(t+2)\eta_t
        \ls\E\norm{u_t-x^*}^2-\E\norm{u_{t+1}-x^*}^2\rs \\
      +9\max\lb M,L_2\rb D^3\frac{(t+1)(t+2)}{(t+3)^2}
      +9\sigma_2D^2\frac{(t+1)(t+2)}{t+3}
      +\frac{3\sigma_1^2}{2}\frac{(t+1)(t+2)}{\eta_t}.
  \end{gather*}
  Now sum it over $t=0,\ldots,T-1$. The left-hand side telescopes:
  \begin{gather*}
    \sum_{t=0}^{T-1}\ls
      (t+1)(t+2)(t+3)\E\lp f(x_{t+1})-f^*\rp
      -t(t+1)(t+2)\E\lp f(x_t)-f^*\rp
    \rs\\
    =T(T+1)(T+2)\E\lp f(x_T)-f^*\rp,
  \end{gather*}
  where we used $\alpha_0=1$. Next,
  \begin{equation}
    \label{eq:basic_summed}
    \begin{aligned}
      T(T+1)(T+2)\E\lp f(x_T)-f^*\rp
      \leq{}&\frac{3}{2}\sum_{t=0}^{T-1}(t+1)(t+2)\eta_t
        \ls\E\norm{u_t-x^*}^2-\E\norm{u_{t+1}-x^*}^2\rs \\
      &+9\max\lb M,L_2\rb D^3\sum_{t=0}^{T-1}\frac{(t+1)(t+2)}{(t+3)^2}\\
      &+9\sigma_2D^2\sum_{t=0}^{T-1}\frac{(t+1)(t+2)}{t+3}
      +\frac{3\sigma_1^2}{2}\sum_{t=0}^{T-1}\frac{(t+1)(t+2)}{\eta_t}.
    \end{aligned}
  \end{equation}
  Let us consider the first sum in inequality above. 
  \begin{gather*}
    \sum_{t=0}^{T-1}(t+1)(t+2)\eta_t
      \ls\E\norm{u_t-x^*}^2-\E\norm{u_{t+1}-x^*}^2\rs \\
    =2\eta_0\E\norm{u_0-x^*}^2
    +\sum_{t=0}^{T-2}\ls(t+2)(t+3)\eta_{t+1}-(t+1)(t+2)\eta_t\rs
      \E\norm{u_{t+1}-x^*}^2\\
    -T(T+1)\eta_{T-1}\E\norm{u_T-x^*}^2 \\
    \leq D^2\lp2\eta_0+
      \sum_{t=0}^{T-2}\ls(t+2)(t+3)\eta_{t+1}-(t+1)(t+2)\eta_t\rs\rp \\
   =T(T+1)\eta_{T-1}D^2 \stackrel{\eqref{eq:basic_reg}}{=}\sigma_1D T(T+1)\sqrt T.
  \end{gather*}
  The remaining sums satisfy
  \begin{gather*}
    \sum_{t=0}^{T-1}\frac{(t+1)(t+2)}{(t+3)^2}\leq T,
    \qquad
    \sum_{t=0}^{T-1}\frac{(t+1)(t+2)}{t+3}\leq\sum_{t=0}^{T-1}(t+1)=\frac{T(T+1)}{2}, \\
    \sum_{t=0}^{T-1}\frac{(t+1)(t+2)}{\eta_t}
    \stackrel{\eqref{eq:basic_reg}}{=}\frac{D}{\sigma_1}\sum_{t=0}^{T-1}(t+2)\sqrt{t+1} 
    \leq\frac{D\sqrt T}{\sigma_1}\sum_{t=0}^{T-1}(t+2)
    =\frac{D\sqrt T T(T+3)}{2\sigma_1}.
  \end{gather*}
  Substituting into~\eqref{eq:basic_summed} and dividing by $T(T+1)(T+2)$, we obtain
  \begin{gather*}
    \E\lp f(x_T)-f^*\rp
    \leq\frac{9\max\lb M,L_2\rb D^3}{(T+1)(T+2)}
    +\frac{9\sigma_2D^2}{2(T+2)}
    +\frac{3\sigma_1D\sqrt T}{2(T+2)}
    +\frac{3\sigma_1D\sqrt T(T+3)}{4(T+1)(T+2)} \\
    =\frac{9\max\lb M,L_2\rb D^3}{(T+1)(T+2)}
    +\frac{9\sigma_2D^2}{2(T+2)}
    +\frac{3\sigma_1D\sqrt T(3T+5)}{4(T+1)(T+2)}.
  \end{gather*}
  Since $T(3T+5)\leq3(T+1)(T+2)$, this proves~\eqref{eq:basic_rate}.

  If $\sigma_1=0$, then $\eta_t=0$ and $g_t=\g(x_t)$. The gradient error and quadratic regularization terms in~\eqref{eq:basic_one_step} vanish. The same summation then proves~\eqref{eq:basic_rate} without its last term. 
\end{proof}

\subsection{Proof of Proposition~\ref{prop:weak_convergence_guarantee}}

\begin{proof}
  Let $\Delta_t=f(x_t)-f^*$. Applying Lemma~\ref{lem:basic_one_step} with $M=\eta_t=0$ and using Cauchy--Schwarz,~\eqref{eq:stoch_grad}, and~\eqref{eq:stoch_hess} gives
  \begin{equation*}
    \E\Delta_{t+1}
    \leq(1-\alpha_t)\E\Delta_t+\alpha_t\sigma_1D
    +\alpha_t^2\sigma_2D^2+\frac{L_2}{3}\alpha_t^3D^3.
  \end{equation*}
  Multiplying by $(t+1)(t+2)(t+3)$, substituting $\alpha_t=3/(t+3)$, and summing over $t=0,\ldots,T-1$ yields
  \begin{gather*}
    T(T+1)(T+2)\E\Delta_T\\
    \leq3\sigma_1D\sum_{t=0}^{T-1}(t+1)(t+2)
    +9\sigma_2D^2\sum_{t=0}^{T-1}\frac{(t+1)(t+2)}{t+3} 
    +9L_2D^3\sum_{t=0}^{T-1}\frac{(t+1)(t+2)}{(t+3)^2} \\
    \leq T(T+1)(T+2)\sigma_1D
    +\frac{9T(T+1)}{2}\sigma_2D^2+9TL_2D^3,
  \end{gather*}
  where the initial term vanishes since $\alpha_0=1$. Dividing by $T(T+1)(T+2)$ gives
  \begin{equation*}
    \E\lp f(x_T)-f^*\rp
    \leq\sigma_1D+\frac{9\sigma_2D^2}{2(T+2)}
    +\frac{9L_2D^3}{(T+1)(T+2)},
  \end{equation*}
  which proves~\eqref{eq:weak_convergence_guarantee}.
\end{proof}

\subsection{Proof of First-order Reduction}
\label{sec:proof_first_order}

\begin{proof}[Proof of Proposition~\ref{prop:first_order_reduction}]
By
$L_1$-smoothness and $x_{t+1}-x_t=\alpha_t(u_{t+1}-x_t)$,
\begin{equation*}
  f(x_{t+1})
  \leq f(x_t)+\alpha_t\la\g(x_t),u_{t+1}-x_t\ra
  +\frac{L_1\alpha_t^2}{2}\norm{u_{t+1}-x_t}^2.
\end{equation*}
Combining this inequality with
\eqref{eq:basic_model_comparison} at $x=x^*$, convexity of
$f$, and $\norm{u_{t+1}-x_t}\leq D$ gives
\eqref{eq:basic_one_step} with its Hessian-error and cubic
terms replaced by $L_1\alpha_t^2D^2/2$.

Suppose first that $\sigma_1>0$.
The conditional-unbiasedness and Young's inequality argument
leading to~\eqref{eq:basic_expected_step} is unchanged.
Repeating the weighted summation and the bounds following
\eqref{eq:basic_summed}, the gradient-noise and regularization
terms contribute at most $9\sigma_1D/(4\sqrt T)$. The remaining
smoothness term gives
\begin{equation*}
    \E\lp f(x_T)-f^*\rp
    \leq \frac{9\sigma_1D}{4\sqrt T}
      +\frac{9L_1D^2}{2T(T+1)(T+2)}
        \sum_{t=0}^{T-1}\frac{(t+1)(t+2)}{t+3}
    \leq \frac{9\sigma_1D}{4\sqrt T}
      +\frac{9L_1D^2}{4(T+2)},
\end{equation*}
where the last inequality uses
$\sum_{t=0}^{T-1}(t+1)(t+2)/(t+3)\leq T(T+1)/2$.
If $\sigma_1=0$, then $\eta_t=0$ and $g_t=\g(x_t)$. The model subproblem is then linear, the gradient-error
and regularization terms vanish, and the same summation proves
the bound without its first term.
\end{proof}

\section{Solving the subproblem}
\label{sec:subproblem_solution}

In this section, we discuss how to solve the subproblem of the contracting step.

\subsection{General feasible set}

The subproblem is the minimization of the model $\omega^M_{t}(u)$ from \eqref{eq:cubic_model} over the set $\XCal$:
\begin{equation}
    \label{eq:subproblem}
    u_{t+1} = \argmin_{u\in \XCal} \lb \la g(x_{t}),u-x_{t}\ra 
          +\tfrac{\alpha_t}{2}\la H(x_{t})(u-x_{t}),u-x_{t}\ra\\
           + \tfrac{\eta_t
          }{2}\|u - u_{t}\|^2 + \tfrac{M \alpha_t^2}{6}\norm{u-x_{t}}^3 \rb
\end{equation}
 $u_{t+1}$ is characterized by the optimality condition
\begin{equation}
    \label{eq:subproblem_optimality}
    \begin{aligned}
    &\left\langle
    g(x_{t})
          +\alpha_t  H(x_{t})(u_{t+1}-x_{t})
            + \eta_t
          (u_{t+1} - u_{t})\right.\\
    &\qquad\left.
          + \frac{M \alpha_t^2}{2}\norm{u_{t+1}-x_{t}}(u_{t+1}-x_t), y-u_{t+1}\right\rangle
          \geq 0, \qquad \forall y\in\XCal.
    \end{aligned}
\end{equation}
Writing $u_{t+1}-u_t = (u_{t+1}-x_t) + (x_t-u_t)$, we can rewrite \eqref{eq:subproblem_optimality} as
\begin{equation}
    \label{eq:subproblem_optimality_V2}
    \la g(x_{t}) + \eta_t
          (x_{t} - u_{t})
          +\ls \alpha_t  H(x_{t})
           + \eta_t I\rs
          (u_{t+1} - x_{t}) + \tfrac{M \alpha_t^2}{2}\norm{u_{t+1}-x_{t}}(u_{t+1}-x_t), y-u_{t+1} \ra \geq 0,
\end{equation}
This is exactly the optimality condition of the problem
\begin{equation}
    \label{eq:subproblem_analog}
    \begin{aligned}
    u_{t+1} = \argmin_{u\in \XCal} \left\{\right.&\la g(x_{t}) + \eta_t
          (x_{t} - u_{t}),u-x_{t}\ra 
          +\frac{1}{2}\la (\alpha_t  H(x_{t})
           + \eta_t I)(u-x_{t}),u-x_{t}\ra\\
          &\left.+ \frac{M \alpha_t^2}{6}\norm{u-x_{t}}^3 \right\}.
    \end{aligned}
\end{equation}
The same conclusion also follows directly from the identity
$\tfrac{\eta_t}{2}\norm{u-u_t}^2 = \tfrac{\eta_t}{2}\norm{u-x_t}^2 + \eta_t\langle x_t-u_t, u-x_t\rangle + \tfrac{\eta_t}{2}\norm{x_t-u_t}^2$:
the objectives of \eqref{eq:subproblem} and \eqref{eq:subproblem_analog} differ only by a constant.

Thus, \eqref{eq:subproblem} is a classical cubic regularized Newton subproblem on $\XCal$. It has the shifted gradient $g(x_t)+\eta_t(x_t-u_t)$, the Hessian $\alpha_t H(x_t)+\eta_t I$, and the regularization parameter $M\alpha_t^2$. Therefore, solving \eqref{eq:subproblem} has the same complexity as solving the classical cubic subproblem on $\XCal$. Moreover, the additional term $\eta_t I$ makes the quadratic part strongly convex, which can only improve the conditioning of the problem. 

\subsection{Euclidean ball}

Now we consider the case where $\XCal$ is a Euclidean ball of diameter $D$ centered at $x_c$, that is, $\XCal = \{x : \norm{x-x_c}\leq \tfrac{D}{2}\}$. We write the constraint in the smooth form $\norm{u-x_c}^2 \leq \tfrac{D^2}{4}$, so \eqref{eq:subproblem} becomes
\begin{equation}
    \label{eq:subproblem_ball}
    u_{t+1} = \argmin_{u:\, \norm{u-x_c}^2 \leq D^2/4} \omega^M_{t}(u).
\end{equation}
For $\lambda \geq 0$, the Lagrangian function is
\begin{equation}
    \label{eq:lagrangian_ball}
    L(u,\lambda) = \omega^M_{t}(u) + \lambda\Big(\norm{u-x_c}^2 - \tfrac{D^2}{4}\Big).
\end{equation}
Problem \eqref{eq:subproblem_ball} is convex, and Slater's condition holds. Therefore, $u_{t+1}$ is optimal if and only if there exists $\lambda$ such that the KKT conditions hold:
\begin{align}
    & g(x_t) + \alpha_t H(x_t)(u_{t+1}-x_t) + \eta_t(u_{t+1}-u_t)  + \tfrac{M\alpha_t^2}{2}\norm{u_{t+1}-x_t}(u_{t+1}-x_t) + 2\lambda(u_{t+1}-x_c) = 0,
      \label{eq:subproblem_optimality_ball}\\
    & \lambda\Big(\norm{u_{t+1}-x_c}^2 - \tfrac{D^2}{4}\Big) = 0, \qquad
      \lambda \geq 0, \qquad
      \norm{u_{t+1}-x_c} \leq \tfrac{D}{2}.
      \label{eq:subproblem_optimality_ball2}
\end{align}
For a fixed $\lambda \geq 0$, let $u(\lambda)$ be the unconstrained minimizer of $L(\cdot,\lambda)$. In other words, $u(\lambda)$ solves \eqref{eq:subproblem_optimality_ball} with $u_{t+1}$ replaced by $u(\lambda)$. Let $r = \norm{u(\lambda)-x_t}$. Rearranging \eqref{eq:subproblem_optimality_ball}, we get
\begin{equation}
    \label{eq:u_lambda}
    u(\lambda) = x_t - \Big(\alpha_t H(x_t) + \big(\eta_t + 2\lambda + \tfrac{M\alpha_t^2}{2}r\big) I\Big)^{-1}
    \big(g(x_t) + \eta_t(x_t-u_t) + 2\lambda(x_t-x_c)\big).
\end{equation}
Since $r$ depends on $u(\lambda)$, computing $u(\lambda)$ requires a one-dimensional search over $r \geq 0$. This is exactly the same procedure as for the unconstrained cubic subproblem.

We solve \eqref{eq:subproblem_ball} in two steps. First, we compute $u(0)$, which is the solution of the unconstrained problem. If $\norm{u(0)-x_c} \leq \tfrac{D}{2}$, then $u_{t+1} = u(0)$. Otherwise, the constraint is active, and we search for $\lambda > 0$ such that $\norm{u(\lambda)-x_c} = \tfrac{D}{2}$. Such a $\lambda$ always exists and can be found by bisection. The reason is that the function $\lambda \mapsto \norm{u(\lambda)-x_c}$ is continuous and non-increasing, and $u(\lambda) \to x_c$ as $\lambda \to \infty$.

Hence, the ball-constrained subproblem requires two nested one-dimensional searches: an outer search over $\lambda$ and an inner search over $r$. This matches the complexity of the classical cubic regularized Newton method on the same set. Moreover, our method can also be used with $M = 0$. In this case, the matrix in \eqref{eq:u_lambda} does not depend on $r$, so the inner search is not needed. The unconstrained step then requires a single linear solve, and the ball-constrained step becomes a trust-region-type subproblem with only one search over $\lambda$.

\section{Proof of Accelerated Method}
\label{sec:proof_optimal}

We restate Algorithm~\ref{alg:optimal_method} as Algorithm~\ref{alg:optimal_method_appendix}, including the weights at $t=1$. For $t\geq1$, the parameters are

\begin{equation}
  \label{eq:optimal_reg}
  \lambda_t=\frac{\sigma_1t^3}{D^2}+\frac{\sigma_2t^{3/2}}{D}+2M,
  \qquad \alpha_t=\frac{4}{t+3},
  \qquad \beta_t=\frac{\lambda_tD}{t^{3/2}}.
\end{equation}
The cubic model is
\begin{equation}
  \label{eq:optimal_model}
    \Omega_t(u)= \la g_t,u-z_{t-1}\ra
    +\frac{\alpha_t}{2}\la H_t(u-z_{t-1}),u-z_{t-1}\ra
    +\frac{\alpha_t\beta_t}{2}\norm{u-z_{t-1}}^2
    +\frac{\lambda_t\alpha_t^2}{3}\norm{u-z_{t-1}}^3.
\end{equation}

\begin{algorithm}[H]
  \caption{Accelerated Stochastic Contracting Cubic Newton}
  \label{alg:optimal_method_appendix}
  \begin{algorithmic}[1]
    \STATE \textbf{Input:} $w_0=z_0=x_0\in\XCal$, $A_0=0$, $M>0$ with $M\geq L_2$.
    \STATE \textbf{Parameters:} $\lambda_t,\alpha_t,\beta_t$ given by~\eqref{eq:optimal_reg}.
    \FOR{$t=1,2,\ldots$}
      \STATE Set $v_t=(1-\alpha_t)w_{t-1}+\alpha_tz_{t-1}$.
      \STATE Sample $\xi_t$ and set $g_t=g(v_t,\xi_t)$, $H_t=H(v_t,\xi_t)$.
      \STATE Compute
      \begin{equation}
        \label{eq:optimal_predictor}
        u_t=\argmin_{u\in\XCal}\Omega_t(u),
        \qquad y_t=(1-\alpha_t)w_{t-1}+\alpha_tu_t.
      \end{equation}
      \STATE Set
      \begin{equation}
        \label{eq:optimal_damping}
        \tau_t=\beta_t+\lambda_t\norm{y_t-v_t}.
      \end{equation}
      \STATE Set
      \begin{equation}
        \label{eq:optimal_weights}
        a_t=
        \begin{cases}
          \frac{\lambda_1}{\tau_1}, & t=1,\\[4pt]
          \min\lb\frac{\alpha_tA_{t-1}}{1-\alpha_t},
                    \frac{\lambda_t}{\alpha_t\tau_t}\rb, & t\geq2,
        \end{cases}
        \qquad A_t=A_{t-1}+a_t.
      \end{equation}
      \STATE Sample $\widetilde\xi_t$ and set $\widetilde g_t=g(y_t,\widetilde\xi_t)$.
      \STATE Set
      \begin{equation}
        \label{eq:optimal_updates}
        \begin{aligned}
          w_t&=\frac{A_{t-1}w_{t-1}+a_tu_t}{A_t},\\
          z_t&=\argmin_{z\in\XCal}\lb\la\widetilde g_t,z\ra
          +\frac{\lambda_t}{2a_t}\norm{z-z_{t-1}}^2\rb.
        \end{aligned}
      \end{equation}
    \ENDFOR
  \end{algorithmic}
\end{algorithm}

For the analysis, set $\lambda_0=\lambda_1$. Since $\alpha_1=1$, the first iteration satisfies
\begin{equation}
  \label{eq:optimal_initialization}
    v_1 =x_0,\qquad y_1=w_1=u_1, \qquad
    \tau_1=\beta_1+\lambda_1\norm{u_1-x_0},
    \qquad a_1=A_1=\frac{\lambda_1}{\tau_1}.
\end{equation}

Let $\mathcal F_{t-1}$ and $\mathcal G_t$ denote the histories before sampling $\xi_t$ and $\widetilde\xi_t$, respectively. Assume that $g_t$ and $H_t$ satisfy Assumptions~\ref{as:stoch_grad} and~\ref{as:stoch_hess}, respectively, and that $\widetilde g_t$ satisfies Assumptions~\ref{as:stoch_grad} and~\ref{as:unbiased_grad}. By these assumptions and fresh sampling, we have
\begin{gather*}
  \E\lp\norm{g_t-\g(v_t)}^2\mid\mathcal F_{t-1}\rp\leq\sigma_1^2,
  \qquad \E\lp\norm{H_t-\h(v_t)}^2\mid\mathcal F_{t-1}\rp\leq\sigma_2^2,\\
  \E\lp\widetilde g_t\mid\mathcal G_t\rp=\g(y_t),
  \qquad \E\lp\norm{\widetilde g_t-\g(y_t)}^2\mid\mathcal G_t\rp\leq\sigma_1^2.
\end{gather*}

For $t\geq1$, let $s_t=y_t-v_t$.

\begin{lemma}
  \label{lem:optimal_geometry}
  For the iterates of Algorithm~\ref{alg:optimal_method_appendix} and every $t\geq1$,
  \begin{equation}
    \label{eq:optimal_geometry}
    s_t=\alpha_t(u_t-z_{t-1}),
    \qquad \norm{s_t}\leq\alpha_tD.
  \end{equation}
  The weights satisfy
  \begin{equation}
    \label{eq:optimal_cap}
    0<a_t\leq\alpha_tA_t,
    \qquad 0<a_t\alpha_t\tau_t\leq\lambda_t.
  \end{equation}
  Consequently,
  \begin{equation}
    \label{eq:optimal_averaging}
    w_t=\ls1-\frac{a_t}{\alpha_tA_t}\rs w_{t-1}
    +\frac{a_t}{\alpha_tA_t}y_t.
  \end{equation}
  Moreover,
  \begin{equation}
    \label{eq:optimal_model_condition}
    \la g_t+H_ts_t+\tau_ts_t,x-u_t\ra\geq0,
    \quad x\in\XCal.
  \end{equation}
\end{lemma}

\begin{proof}
  By the definitions of $v_t$ and $y_t$,
  \begin{gather*}
    s_t=(1-\alpha_t)w_{t-1}+\alpha_tu_t
    -\ls(1-\alpha_t)w_{t-1}+\alpha_tz_{t-1}\rs
    =\alpha_t(u_t-z_{t-1}),\\
    \norm{s_t}=\alpha_t\norm{u_t-z_{t-1}}\leq\alpha_tD.
  \end{gather*}
  For $t\geq2$,~\eqref{eq:optimal_weights} gives $a_t>0$ and
  \begin{gather*}
    a_t(1-\alpha_t)\leq\alpha_tA_{t-1},
    \qquad a_t\alpha_t\tau_t\leq\lambda_t,
    \qquad a_t=a_t(1-\alpha_t)+\alpha_ta_t
    \leq\alpha_t(A_{t-1}+a_t)=\alpha_tA_t.
  \end{gather*}
  At $t=1$,~\eqref{eq:optimal_initialization} gives $a_1=\alpha_1A_1>0$ and $a_1\alpha_1\tau_1=\lambda_1$. Thus~\eqref{eq:optimal_cap} holds for every $t\geq1$. By~\eqref{eq:optimal_updates} and~\eqref{eq:optimal_predictor},
  \begin{equation*}
    A_tw_t=A_{t-1}w_{t-1}+a_tu_t
    =(A_t-a_t)w_{t-1}
    +\frac{a_t}{\alpha_t}\ls y_t-(1-\alpha_t)w_{t-1}\rs
    =\ls A_t-\frac{a_t}{\alpha_t}\rs w_{t-1}
    +\frac{a_t}{\alpha_t}y_t,
  \end{equation*}
  which proves~\eqref{eq:optimal_averaging}. Finally,
  \begin{gather*}
    \nabla\Omega_t(u_t)
    \stackrel{\eqref{eq:optimal_model}}{=}
    g_t+\alpha_tH_t(u_t-z_{t-1})
    +\alpha_t\beta_t(u_t-z_{t-1})
    +\lambda_t\alpha_t^2\norm{u_t-z_{t-1}}(u_t-z_{t-1})\\
    =g_t+H_ts_t+\ls\beta_t+\lambda_t\norm{s_t}\rs s_t
    \stackrel{\eqref{eq:optimal_damping}}{=}g_t+H_ts_t+\tau_ts_t.
  \end{gather*}
  Optimality of $u_t$ over $\XCal$ gives~\eqref{eq:optimal_model_condition}.
\end{proof}

Define the potential
\begin{equation}
  \label{eq:optimal_energy}
  E_t=A_t(f(w_t)-f^*)+\frac{\lambda_t}{2}\norm{z_t-x^*}^2,
  \quad t\geq0.
\end{equation}
Since $A_0 = 0, ~\lambda_0 = \lambda_1$, and $z_0 = x_0$ 
\begin{equation*}
  E_0=\frac{\lambda_1}{2}\norm{x_0-x^*}^2\leq\frac{\lambda_1D^2}{2}.
\end{equation*}

For every integer $T\geq1$, the expected objective gap satisfies
\begin{equation}
  \label{eq:goal}
  \E\lp f(w_T)-f^*\rp
  \stackrel{\eqref{eq:optimal_energy}}{\leq}\E\lp E_TA_T^{-1}\rp
  \leq
  \ls\E E_T^2\rs^{1/2}\ls\E A_T^{-2}\rs^{1/2}.
\end{equation}
Since $A_T$ is random, we bound the second moments of both factors. Our goal is to prove
\begin{equation*}
  \ls\E E_T^2\rs^{1/2}=\cO(\lambda_TD^2),
  \qquad
  \ls\E A_T^{-2}\rs^{1/2}=\cO\ls\frac{D}{T^{7/2}}\rs.
\end{equation*}

\subsection{One Step Progress}
\label{subsec:one_step}

We first establish a one-step bound for the potential and use it to control its second moment.

\begin{lemma}
  \label{lem:optimal_one_step}
  For every $t\geq1$,
  \begin{equation}
    \label{eq:optimal_one_step}
    \begin{aligned}
      E_t+\frac{a_t\tau_t}{8\alpha_t}\norm{s_t}^2
      \leq{}&E_{t-1}+\frac{D^2}{2}(\lambda_t-\lambda_{t-1})\\
      &+\min\lb
        \begin{aligned}
          &\frac{4}{\alpha_t^2\lambda_t}
          \ls\norm{H_t-\h(v_t)}^2+\frac{t^3}{D^2}\norm{g_t-\g(v_t)}^2\rs,\\
          &\frac{D}{\alpha_t}\norm{H_t-\h(v_t)}
          +\frac{t^{3/2}}{\alpha_t}\norm{g_t-\g(v_t)}
        \end{aligned}
      \rb\\
      &+\min\lb\frac{2t^3\norm{\widetilde g_t-\g(y_t)}^2}{\alpha_t^2\lambda_tD^2},
                  \frac{t^{3/2}\norm{\widetilde g_t-\g(y_t)}}{\alpha_t}\rb
      -a_t\la\widetilde g_t-\g(y_t),z_{t-1}-x^*\ra.
    \end{aligned}
  \end{equation}
\end{lemma}

\begin{proof}

By~\eqref{eq:optimal_cap}, $0<a_t/(\alpha_tA_t)\leq1$. Thus, by~\eqref{eq:optimal_averaging} and convexity,
\begin{gather*}
  A_tf(w_t)
  =A_tf\ls\ls1-\frac{a_t}{\alpha_tA_t}\rs w_{t-1}
  +\frac{a_t}{\alpha_tA_t}y_t\rs
  \leq\ls A_t-\frac{a_t}{\alpha_t}\rs f(w_{t-1})
  +\frac{a_t}{\alpha_t}f(y_t).
\end{gather*}
Using $A_t-A_{t-1}=a_t$, we obtain
\begin{gather*}
  A_t(f(w_t)-f^*)-A_{t-1}(f(w_{t-1})-f^*)\\
  \leq\ls A_t-\frac{a_t}{\alpha_t}-A_{t-1}\rs f(w_{t-1})
  +\frac{a_t}{\alpha_t}f(y_t)-a_tf^*\\
  =\frac{a_t}{\alpha_t}
  \ls f(y_t)-(1-\alpha_t)f(w_{t-1})-\alpha_tf^*\rs\\
  \leq\frac{a_t}{\alpha_t}
  \la\g(y_t),y_t-(1-\alpha_t)w_{t-1}-\alpha_tx^*\ra
  =a_t\la\g(y_t),u_t-x^*\ra.
\end{gather*}

  Hence,
  \begin{equation}
    \label{eq:optimal_convexity}
    A_t(f(w_t)-f^*)\leq A_{t-1}(f(w_{t-1})-f^*)
    +a_t\la\g(y_t),u_t-x^*\ra.
  \end{equation}

  The optimality condition for~\eqref{eq:optimal_updates}, multiplied by $a_t>0$, gives
  \begin{equation*}
    \la a_t\widetilde g_t+\lambda_t(z_t-z_{t-1}),x^*-z_t\ra\geq0.
  \end{equation*}
  Since
  \begin{equation*}
    \norm{z_{t-1}-x^*}^2
    =\norm{z_t-x^*}^2+\norm{z_t-z_{t-1}}^2
    -2\la z_t-z_{t-1},z_t-x^*\ra,
  \end{equation*}
  we obtain
  \begin{gather*}
    a_t\la\widetilde g_t,z_t-x^*\ra
    \leq-\lambda_t\la z_t-z_{t-1},z_t-x^*\ra
    =\frac{\lambda_t}{2}
    \ls\norm{z_{t-1}-x^*}^2-\norm{z_t-x^*}^2-\norm{z_t-z_{t-1}}^2\rs.
  \end{gather*}
  Rearranging,
  \begin{equation}
    \label{eq:three_dist_bound}
    \frac{\lambda_t}{2}
    \ls \norm{z_t-x^*}^2 - \norm{z_{t-1}-x^*}^2 \rs \leq a_t\la\widetilde g_t,x^* - z_t \ra -\frac{\lambda_t}{2}\norm{z_t-z_{t-1}}^2.
  \end{equation}
  Combining this bound with~\eqref{eq:optimal_convexity}, $\lambda_t\geq\lambda_{t-1}$, and $\norm{z_{t-1}-x^*}\leq D$, we obtain
  \begin{gather*}
    E_t-E_{t-1}
    \stackrel{\eqref{eq:optimal_energy}}{=}
    A_t(f(w_t)-f^*)-A_{t-1}(f(w_{t-1})-f^*)\\
    +\frac{\lambda_t}{2}\ls\norm{z_t-x^*}^2-\norm{z_{t-1}-x^*}^2\rs
    +\frac{\lambda_t-\lambda_{t-1}}{2}\norm{z_{t-1}-x^*}^2\\
    \stackrel{\eqref{eq:optimal_convexity}}{\leq}a_t\la\g(y_t),u_t-x^*\ra
    +\frac{\lambda_t}{2}\ls\norm{z_t-x^*}^2-\norm{z_{t-1}-x^*}^2\rs\\
    +\frac{\lambda_t-\lambda_{t-1}}{2}\norm{z_{t-1}-x^*}^2 \\
    \stackrel{\eqref{eq:three_dist_bound}}{\leq} 
    a_t\la\g(y_t),u_t-z_t\ra
    -a_t\la\widetilde g_t-\g(y_t),z_t-x^*\ra 
    -\frac{\lambda_t}{2}\norm{z_t-z_{t-1}}^2
    + \frac{D^2}{2}(\lambda_t-\lambda_{t-1}).
  \end{gather*}
  Applying~\eqref{eq:optimal_model_condition} at $x=z_t$ gives
  \begin{equation*}
    \la\g(y_t),u_t-z_t\ra
    \leq\la\g(y_t)-g_t-H_ts_t-\tau_ts_t,u_t-z_t\ra.
  \end{equation*}
  Consequently,
  \begin{equation}
    \label{eq:optimal_energy_projection}
    \begin{aligned}
      E_t-E_{t-1}\leq{}&\frac{D^2}{2}(\lambda_t-\lambda_{t-1})
      -\frac{\lambda_t}{2}\norm{z_t-z_{t-1}}^2\\
      &+a_t\la\g(y_t)-g_t-H_ts_t-\tau_ts_t,u_t-z_t\ra
      -a_t\la\widetilde g_t-\g(y_t),z_t-x^*\ra.
    \end{aligned}
  \end{equation}

  We first bound the term containing $\widetilde g_t-\g(y_t)$. Splitting $z_t-x^*=(z_{t-1}-x^*)+(z_t-z_{t-1})$, we get
  \begin{gather*}
    -a_t\la\widetilde g_t-\g(y_t),z_t-x^*\ra
    -\frac{\lambda_t}{8}\norm{z_t-z_{t-1}}^2\\
    =-a_t\la\widetilde g_t-\g(y_t),z_{t-1}-x^*\ra
    -a_t\la\widetilde g_t-\g(y_t),z_t-z_{t-1}\ra
    -\frac{\lambda_t}{8}\norm{z_t-z_{t-1}}^2.
  \end{gather*}
  Young's inequality and $\norm{z_t-z_{t-1}}\leq D$ give, respectively,
  \begin{gather*}
    -a_t\la\widetilde g_t-\g(y_t),z_t-z_{t-1}\ra
    -\frac{\lambda_t}{8}\norm{z_t-z_{t-1}}^2
    \leq\frac{2a_t^2}{\lambda_t}\norm{\widetilde g_t-\g(y_t)}^2,\\
    -a_t\la\widetilde g_t-\g(y_t),z_t-z_{t-1}\ra
    -\frac{\lambda_t}{8}\norm{z_t-z_{t-1}}^2
    \leq a_tD\norm{\widetilde g_t-\g(y_t)}.
  \end{gather*}
  Taking the smaller of these two bounds, we obtain
  \begin{equation}
    \label{eq:optimal_fresh_gradient_error}
    \begin{aligned}
      &-a_t\langle \widetilde g_t  -\g(y_t),z_t-x^*\rangle
      -\frac{\lambda_t}{8}\norm{z_t-z_{t-1}}^2\\
      &\leq-a_t\la\widetilde g_t-\g(y_t),z_{t-1}-x^*\ra\\
      &\quad+\min\lb\frac{2a_t^2}{\lambda_t}\norm{\widetilde g_t-\g(y_t)}^2,
                  a_tD\norm{\widetilde g_t-\g(y_t)}\rb.
    \end{aligned}
  \end{equation}

  Applying~\eqref{eq:optimal_fresh_gradient_error} in~\eqref{eq:optimal_energy_projection}, we obtain
  \begin{gather*}
    E_t-E_{t-1}
    \leq\frac{D^2}{2}(\lambda_t-\lambda_{t-1})
    +a_t\la\g(y_t)-g_t-H_ts_t-\tau_ts_t,u_t-z_t\ra
    -\frac{3\lambda_t}{8}\norm{z_t-z_{t-1}}^2\\
    -a_t\la\widetilde g_t-\g(y_t),z_{t-1}-x^*\ra
    +\min\lb\frac{2a_t^2}{\lambda_t}\norm{\widetilde g_t-\g(y_t)}^2,
                a_tD\norm{\widetilde g_t-\g(y_t)}\rb.
  \end{gather*}
  It remains to bound
  \begin{equation}
    \label{eq:remaining_bound}
    a_t\la\g(y_t)-g_t-H_ts_t-\tau_ts_t,u_t-z_t\ra
    -\frac{3\lambda_t}{8}\norm{z_t-z_{t-1}}^2.
  \end{equation}
  We rewrite
  \begin{equation*}
    \g(y_t)-g_t-H_ts_t
    =\g(y_t)-\g(v_t)-\h(v_t)s_t
    -(H_t-\h(v_t))s_t-(g_t-\g(v_t)).
  \end{equation*}
  Since $s_t=y_t-v_t$, Lemma~\ref{lem:taylor_remainder} with $x=v_t$ and $y=y_t$, and using $\lambda_t\geq2M\geq2L_2$ and $\tau_t\geq\lambda_t\norm{s_t}$, we obtain
  \begin{equation}
    \label{eq:optimal_taylor}
    \norm{\g(y_t)-\g(v_t)-\h(v_t)s_t}
    \leq\frac{L_2}{2}\norm{s_t}^2
    \leq\frac{\lambda_t}{4}\norm{s_t}^2
    \leq\frac{\tau_t}{4}\norm{s_t}.
  \end{equation}

  For every vector $h\in\R^d$, Young's inequality and~\eqref{eq:optimal_cap} give
  \begin{equation}
    \label{eq:optimal_young_bound}
    \begin{aligned}
     a_t\la h,u_t-z_t\ra-\frac{3\lambda_t}{8}\norm{z_t-z_{t-1}}^2
       &\stackrel{\eqref{eq:optimal_geometry}}{=}
      \frac{a_t}{\alpha_t}\la h,s_t\ra-a_t\la h,z_t-z_{t-1}\ra
      -\frac{3\lambda_t}{8}\norm{z_t-z_{t-1}}^2\\
      &\leq
      \frac{a_t}{\alpha_t}\la h,s_t\ra+\frac{2a_t^2}{3\lambda_t}\norm{h}^2
      \stackrel{\eqref{eq:optimal_cap}}{\leq}
      \frac{a_t}{\alpha_t}\la h,s_t\ra+\frac{2a_t}{3\alpha_t\tau_t}\norm{h}^2.
    \end{aligned}
  \end{equation}
  Substituting $h=(h+\tau_ts_t)-\tau_ts_t$, we obtain
  \begin{equation}
    \label{eq:optimal_shift_bound}
    \begin{aligned}
      \frac{a_t}{\alpha_t}\la h,s_t\ra
      &+\frac{2a_t}{3\alpha_t\tau_t}\norm{h}^2
      =\frac{2a_t}{3\alpha_t\tau_t}\norm{h+\tau_ts_t}^2
      -\frac{a_t}{3\alpha_t}\la h+\tau_ts_t,s_t\ra
      -\frac{a_t\tau_t}{3\alpha_t}\norm{s_t}^2\\
      &\leq
      \ls\frac{2a_t}{3\alpha_t\tau_t}+\frac{a_t}{3\alpha_t\tau_t}\rs\norm{h+\tau_ts_t}^2
      +\ls\frac{a_t\tau_t}{12\alpha_t}-\frac{a_t\tau_t}{3\alpha_t}\rs\norm{s_t}^2
      =\frac{a_t}{\alpha_t\tau_t}\norm{h+\tau_ts_t}^2
      -\frac{a_t\tau_t}{4\alpha_t}\norm{s_t}^2.
    \end{aligned}
  \end{equation}
  Combining~\eqref{eq:optimal_young_bound} and~\eqref{eq:optimal_shift_bound} gives
  \begin{equation}
    \label{eq:optimal_square_bound}
    a_t\la h,u_t-z_t\ra-\frac{3\lambda_t}{8}\norm{z_t-z_{t-1}}^2
    \leq-\frac{a_t\tau_t}{4\alpha_t}\norm{s_t}^2
    +\frac{a_t}{\alpha_t\tau_t}\norm{h+\tau_ts_t}^2.
  \end{equation}

  We apply~\eqref{eq:optimal_square_bound} with two choices of $h$ to bound~\eqref{eq:remaining_bound}.
  \begin{itemize}
    \item Let $h=\g(y_t)-g_t-H_ts_t-\tau_ts_t$. Then~\eqref{eq:optimal_square_bound} gives
    \begin{gather*}
      a_t\la\g(y_t)-g_t-H_ts_t-\tau_ts_t,u_t-z_t\ra
      -\frac{3\lambda_t}{8}\norm{z_t-z_{t-1}}^2\\
      \leq-\frac{a_t\tau_t}{4\alpha_t}\norm{s_t}^2
      +\frac{a_t}{\alpha_t\tau_t}\norm{\g(y_t)-g_t-H_ts_t}^2.
    \end{gather*}
    By Young's inequality we have
    \begin{gather*}
      \norm{\g(y_t)-g_t-H_ts_t}^2\\
      =\norm{\g(y_t)-\g(v_t)-\h(v_t)s_t
      -(H_t-\h(v_t))s_t-(g_t-\g(v_t))}^2\\
      \leq
      2\norm{\g(y_t)-\g(v_t)-\h(v_t)s_t}^2
      +4\norm{H_t-\h(v_t)}^2\norm{s_t}^2
      +4\norm{g_t-\g(v_t)}^2\\
      \stackrel{\eqref{eq:optimal_taylor}}{\leq}
      \frac{\tau_t^2}{8}\norm{s_t}^2
      +4\norm{H_t-\h(v_t)}^2\norm{s_t}^2
      +4\norm{g_t-\g(v_t)}^2.
    \end{gather*}
    Substituting this bound gives
    \begin{equation}
      \label{eq:optimal_quadratic_error}
      \begin{aligned}
        &a_t\la\g(y_t)-g_t-H_ts_t-\tau_ts_t,u_t-z_t\ra
        -\frac{3\lambda_t}{8}\norm{z_t-z_{t-1}}^2\\
        &\leq-\frac{a_t\tau_t}{8\alpha_t}\norm{s_t}^2
        +\frac{4a_t}{\alpha_t\tau_t}
        \ls\norm{H_t-\h(v_t)}^2\norm{s_t}^2+\norm{g_t-\g(v_t)}^2\rs.
      \end{aligned}
    \end{equation}

    \item Let $h=\g(y_t)-\g(v_t)-\h(v_t)s_t-\tau_ts_t$. Then
    \begin{gather*}
      a_t\la h,u_t-z_t\ra-\frac{3\lambda_t}{8}\norm{z_t-z_{t-1}}^2
      \stackrel{\eqref{eq:optimal_square_bound}}{\leq}
      -\frac{a_t\tau_t}{4\alpha_t}\norm{s_t}^2
      +\frac{a_t}{\alpha_t\tau_t}\norm{\g(y_t)-\g(v_t)-\h(v_t)s_t}^2\\
      \stackrel{\eqref{eq:optimal_taylor}}{\leq}
      -\frac{a_t\tau_t}{4\alpha_t}\norm{s_t}^2
      +\frac{a_t}{\alpha_t\tau_t}\frac{\tau_t^2}{16}\norm{s_t}^2
      =-\frac{3a_t\tau_t}{16\alpha_t}\norm{s_t}^2
      \leq-\frac{a_t\tau_t}{8\alpha_t}\norm{s_t}^2.
    \end{gather*}
    Since
    \begin{equation*}
      \g(y_t)-g_t-H_ts_t-\tau_ts_t
      =h-(H_t-\h(v_t))s_t-(g_t-\g(v_t)),
    \end{equation*}
    and $\norm{u_t-z_t}\leq D$, we also have
    \begin{equation}
      \label{eq:optimal_linear_error}
      \begin{aligned}
        &a_t\la\g(y_t)-g_t-H_ts_t-\tau_ts_t,u_t-z_t\ra
        -\frac{3\lambda_t}{8}\norm{z_t-z_{t-1}}^2\\
        &\leq-\frac{a_t\tau_t}{8\alpha_t}\norm{s_t}^2
        +a_tD\ls\norm{H_t-\h(v_t)}\norm{s_t}+\norm{g_t-\g(v_t)}\rs.
      \end{aligned}
    \end{equation}
  \end{itemize}

  Taking the smaller right-hand side of~\eqref{eq:optimal_quadratic_error} and~\eqref{eq:optimal_linear_error}, and combining with~\eqref{eq:optimal_fresh_gradient_error} in~\eqref{eq:optimal_energy_projection}, we obtain
  \begin{equation}
    \label{eq:optimal_energy_combined}
    \begin{aligned}
      E_t+\frac{a_t\tau_t}{8\alpha_t}\norm{s_t}^2
      \leq{}&E_{t-1}+\frac{D^2}{2}(\lambda_t-\lambda_{t-1})\\
      &+\min\lb
        \begin{aligned}
          &\frac{4a_t}{\alpha_t\tau_t}
          \ls\norm{H_t-\h(v_t)}^2\norm{s_t}^2+\norm{g_t-\g(v_t)}^2\rs,\\
          &a_tD\ls\norm{H_t-\h(v_t)}\norm{s_t}+\norm{g_t-\g(v_t)}\rs
        \end{aligned}
      \rb\\
      &+\min\lb\frac{2a_t^2}{\lambda_t}\norm{\widetilde g_t-\g(y_t)}^2,
                  a_tD\norm{\widetilde g_t-\g(y_t)}\rb\\
      &-a_t\la\widetilde g_t-\g(y_t),z_{t-1}-x^*\ra.
    \end{aligned}
  \end{equation}
  Finally, we bound random coefficients including $\|s_t\|,~\|\tau_t\|,~a_t$  using $\tau_t\geq\lambda_t\norm{s_t}$ and $\tau_t\geq\beta_t$ . First,
  \begin{equation}
    \label{eq:optimal_weight_bound}
    a_t\stackrel{\eqref{eq:optimal_cap}}{\leq}
    \frac{\lambda_t}{\alpha_t\tau_t}
    \leq\frac{\lambda_t}{\alpha_t\beta_t}
    \stackrel{\eqref{eq:optimal_reg}}{=}
    \frac{\lambda_t}{\alpha_t(\lambda_tD/t^{3/2})}
    =\frac{t^{3/2}}{\alpha_tD}.
  \end{equation}
  Consequently,
  \begin{gather*}
    \frac{a_t\norm{s_t}^2}{\alpha_t\tau_t}
    \stackrel{\eqref{eq:optimal_cap}}{\leq}
    \frac{\lambda_t\norm{s_t}^2}{\alpha_t^2\tau_t^2}
    =\frac{1}{\alpha_t^2\lambda_t}\ls\frac{\lambda_t\norm{s_t}}{\tau_t}\rs^2
    \leq\frac{1}{\alpha_t^2\lambda_t},\\
    \frac{a_t}{\alpha_t\tau_t}
    \stackrel{\eqref{eq:optimal_cap}}{\leq}
    \frac{\lambda_t}{\alpha_t^2\tau_t^2}
    \leq\frac{\lambda_t}{\alpha_t^2\beta_t^2}
    \stackrel{\eqref{eq:optimal_reg}}{=}
    \frac{\lambda_t}{\alpha_t^2(\lambda_tD/t^{3/2})^2}
    =\frac{t^3}{\alpha_t^2\lambda_tD^2},\\
    a_tD\norm{s_t}
    \stackrel{\eqref{eq:optimal_cap}}{\leq}
    \frac{\lambda_tD\norm{s_t}}{\alpha_t\tau_t}
    =\frac{D}{\alpha_t}\frac{\lambda_t\norm{s_t}}{\tau_t}
    \leq\frac{D}{\alpha_t},\qquad
    a_tD\stackrel{\eqref{eq:optimal_weight_bound}}{\leq}
    \frac{t^{3/2}}{\alpha_tD}D
    =\frac{t^{3/2}}{\alpha_t},\\
    \frac{2a_t^2}{\lambda_t}
    \stackrel{\eqref{eq:optimal_weight_bound}}{\leq}
    \frac{2}{\lambda_t}\ls\frac{t^{3/2}}{\alpha_tD}\rs^2
    =\frac{2t^3}{\alpha_t^2\lambda_tD^2}.
  \end{gather*}
  Substituting these coefficient bounds into~\eqref{eq:optimal_energy_combined} proves~\eqref{eq:optimal_one_step}.
\end{proof}

\subsection{Potential Bounds}

For the two minimum terms in~\eqref{eq:optimal_one_step}, define
\begin{gather*}
  \delta_t=\min\lb
    \begin{aligned}
      &\frac{4}{\alpha_t^2\lambda_t}
      \ls\norm{H_t-\h(v_t)}^2+\frac{t^3}{D^2}\norm{g_t-\g(v_t)}^2\rs,\\
      &\frac{D}{\alpha_t}\norm{H_t-\h(v_t)}
      +\frac{t^{3/2}}{\alpha_t}\norm{g_t-\g(v_t)}
    \end{aligned}
  \rb\\
  +\min\lb
    \frac{2t^3}{\alpha_t^2\lambda_tD^2}\norm{\widetilde g_t-\g(y_t)}^2,
    \frac{t^{3/2}}{\alpha_t}\norm{\widetilde g_t-\g(y_t)}
  \rb.
\end{gather*}
\begin{lemma}
  \label{lem:optimal_error_moments}
  For every $t\geq1$ and integer $T\geq1$,
  \begin{equation}
    \label{eq:optimal_error_budget}
    \begin{aligned}
      \E\lp\delta_t\mid\mathcal F_{t-1}\rp
      &\leq\sigma_2D\sqrt t+2\sigma_1t^2,\\
      \sum_{i=1}^T\E\lp\delta_i\mid\mathcal F_{i-1}\rp
      &\leq2\lambda_TD^2.
    \end{aligned}
  \end{equation}
  Moreover,
  \begin{equation}
    \label{eq:optimal_error_second_moment}
    \E\lp\delta_t^2\mid\mathcal F_{t-1}\rp
    \leq\frac{3D^2\sigma_2^2}{\alpha_t^2}
    +\frac{6\sigma_1^2t^3}{\alpha_t^2}.
  \end{equation}
\end{lemma}

\begin{proof}
\begin{itemize}
\item We bound first moment. Taking the quadratic bound in each minimum gives
\begin{equation*}
  \delta_t
  \leq\frac{4}{\alpha_t^2\lambda_t}
  \ls\norm{H_t-\h(v_t)}^2+\frac{t^3}{D^2}\norm{g_t-\g(v_t)}^2\rs
  +\frac{2t^3}{\alpha_t^2\lambda_tD^2}\norm{\widetilde g_t-\g(y_t)}^2.
\end{equation*}
Taking conditional expectations and using~\eqref{eq:stoch_grad} and~\eqref{eq:stoch_hess}, with conditioning first on $\mathcal G_t$ for $\widetilde g_t$, we obtain
\begin{equation*}
  \E\lp\delta_t\mid\mathcal F_{t-1}\rp
  \leq\frac{4\sigma_2^2}{\alpha_t^2\lambda_t}
  +\frac{6\sigma_1^2t^3}{\alpha_t^2\lambda_tD^2}.
\end{equation*}
For $t\geq3$, $\alpha_t^{-1}=(t+3)/4\leq t/2$, so
\begin{equation*}
  \E\lp\delta_t\mid\mathcal F_{t-1}\rp
  \leq\frac{\sigma_2^2t^2}{\lambda_t}
  +\frac{3\sigma_1^2t^5}{2\lambda_tD^2}
  \stackrel{\eqref{eq:optimal_reg}}{\leq}
  \sigma_2D\sqrt t+\frac{3}{2}\sigma_1t^2.
\end{equation*}
For $t=1,2$, we use the linear bounds to obtain better constants. Using $\alpha_t^{-1}\leq\sqrt t$ gives
\begin{equation*}
  \E\lp\delta_t\mid\mathcal F_{t-1}\rp
  \leq\frac{\sigma_2D+2\sigma_1t^{3/2}}{\alpha_t}
  \leq\sigma_2D\sqrt t+2\sigma_1t^2.
\end{equation*}
Thus the first bound in~\eqref{eq:optimal_error_budget} holds for every $t\geq1$. By Cauchy--Schwarz and the formula for the sum of squares,
\begin{gather*}
  \sum_{t=1}^T\E\lp\delta_t\mid\mathcal F_{t-1}\rp
  \leq\sigma_2D\sum_{t=1}^T\sqrt t+2\sigma_1\sum_{t=1}^Tt^2
  \leq\sigma_2DT\sqrt{\frac{T+1}{2}}
  +\frac{\sigma_1T(T+1)(2T+1)}{3}\\
  \leq\sigma_2DT^{3/2}+2\sigma_1T^3
  \stackrel{\eqref{eq:optimal_reg}}{\leq}2\lambda_TD^2.
\end{gather*}

\item Now, we bound second moment. Taking the linear bounds and squaring gives
\begin{gather*}
  \delta_t^2
  \leq\ls
    \frac{D}{\alpha_t}\norm{H_t-\h(v_t)}
    +\frac{t^{3/2}}{\alpha_t}\norm{g_t-\g(v_t)}
    +\frac{t^{3/2}}{\alpha_t}\norm{\widetilde g_t-\g(y_t)}
  \rs^2\\
  \leq\frac{3D^2}{\alpha_t^2}\norm{H_t-\h(v_t)}^2
  +\frac{3t^3}{\alpha_t^2}
  \norm{g_t-\g(v_t)}^2+\frac{3t^3}{\alpha_t^2} \norm{\widetilde g_t-\g(y_t)}^2.
\end{gather*}
Taking conditional expectations as above, we obtain
\begin{equation*}
  \E\lp\delta_t^2\mid\mathcal F_{t-1}\rp
  \leq\frac{3D^2\sigma_2^2}{\alpha_t^2}
  +\frac{6\sigma_1^2t^3}{\alpha_t^2}.
\end{equation*}
\end{itemize}
\end{proof}

We now bound $\ls\E E_T^2\rs^{1/2}$, as required in~\eqref{eq:goal}.

\begin{lemma}
  For every integer $T\geq1$,
  \begin{equation}
    \label{eq:optimal_energy_moments}
    \E E_T\leq\frac{5}{2}\lambda_TD^2,
    \qquad
    \ls\E E_T^2\rs^{1/2}\leq5\lambda_TD^2.
  \end{equation}
\end{lemma}

\begin{proof}
  Summing~\eqref{eq:optimal_one_step} and using $E_0\leq\lambda_0D^2/2$, we obtain
  \begin{equation}
    \label{eq:optimal_summed_energy}
      0\leq E_T+\frac{1}{8}\sum_{t=1}^T\frac{a_t\tau_t}{\alpha_t}\norm{s_t}^2
      \leq E_0+\frac{D^2}{2}(\lambda_T-\lambda_0)
      +\sum_{t=1}^T\delta_t+\sum_{t=1}^T\ell_t
      \leq\frac{\lambda_TD^2}{2}
      +\sum_{t=1}^T\delta_t+\sum_{t=1}^T\ell_t,
  \end{equation}
  where $\ell_t=-a_t\la\widetilde g_t-\g(y_t),z_{t-1}-x^*\ra$.

  The points $y_t,z_{t-1}$ and the weight $a_t$ are determined before sampling $\widetilde\xi_t$. Thus,
  \begin{equation}
    \label{eq:optimal_ell_second_moment}
      \E\lp\ell_t^2\mid\mathcal G_t\rp
      \leq a_t^2D^2\E\lp\norm{\widetilde g_t-\g(y_t)}^2\mid\mathcal G_t\rp
      \stackrel{\eqref{eq:stoch_grad}}{\leq}a_t^2D^2\sigma_1^2
      \stackrel{\eqref{eq:optimal_weight_bound}}{\leq}
      \frac{\sigma_1^2t^3}{\alpha_t^2}
      \leq\sigma_1^2t^5.
  \end{equation}
  Moreover,
  \begin{equation*}
    \E\lp\ell_t\mid\mathcal G_t\rp
    =-a_t\la\E\lp\widetilde g_t\mid\mathcal G_t\rp-\g(y_t),z_{t-1}-x^*\ra
    \stackrel{\eqref{eq:unbiased_grad}}{=}0.
  \end{equation*}
  Taking expectations in~\eqref{eq:optimal_summed_energy} gives
  \begin{gather*}
    \E E_T
    \leq\frac{\lambda_TD^2}{2}
    +\sum_{t=1}^T\E\delta_t+\sum_{t=1}^T\E\ell_t
    =\frac{\lambda_TD^2}{2}+\sum_{t=1}^T\E\delta_t
    \stackrel{\eqref{eq:optimal_error_budget}}{\leq}
    \frac{\lambda_TD^2}{2}+2\lambda_TD^2
    =\frac{5}{2}\lambda_TD^2.
  \end{gather*}

  Using~\eqref{eq:optimal_summed_energy} and applying Minkowski's inequality twice, we obtain
  \begin{equation}
    \label{eq:optimal_energy_second_split}
    \ls\E E_T^2\rs^{1/2}
    \leq\frac{\lambda_TD^2}{2}
    +\ls\E\ls\sum_{t=1}^T\delta_t\rs^2\rs^{1/2}
    +\ls\E\ls\sum_{t=1}^T\ell_t\rs^2\rs^{1/2},
  \end{equation}
  since sums have finite second moments by~\eqref{eq:optimal_error_second_moment} and~\eqref{eq:optimal_ell_second_moment}.
  It remains to bound the second moments of the two sums in~\eqref{eq:optimal_energy_second_split}.
   
  For the first sum,
  \begin{equation}
    \label{eq:optimal_error_sum_expansion}
    \E\ls\sum_{t=1}^T\delta_t\rs^2
    =\sum_{t=1}^T\E\delta_t^2
    +2\sum_{1\leq i<j\leq T}\E\lp\delta_i\delta_j\rp.
  \end{equation}
  For $i<j$, $\delta_i$ is determined before iteration $j$. Therefore, by the law of total expectation
  \begin{gather*}
    \E\lp\delta_i\delta_j\rp
    =\E\lp\delta_i\E\lp\delta_j\mid\mathcal F_{j-1}\rp\rp
    \stackrel{\eqref{eq:optimal_error_budget}}{\leq}
    (\sigma_2D\sqrt j+2\sigma_1j^2)\E\delta_i
    \stackrel{\eqref{eq:optimal_error_budget}}{\leq}
    (\sigma_2D\sqrt i+2\sigma_1i^2)
    (\sigma_2D\sqrt j+2\sigma_1j^2).
  \end{gather*}
  By the law of total expectation and~\eqref{eq:optimal_error_second_moment},
  \begin{equation*}
    \E\delta_t^2
    =\E\lp\E\lp\delta_t^2\mid\mathcal F_{t-1}\rp\rp
    \leq\frac{3D^2\sigma_2^2+6\sigma_1^2t^3}{\alpha_t^2}
    \leq3D^2\sigma_2^2t^2+6\sigma_1^2t^5.
  \end{equation*}
  Substituting these bounds into~\eqref{eq:optimal_error_sum_expansion}, we obtain
  \begin{equation}
    \label{eq:optimal_error_sum_moment}
    \begin{aligned}
      \E\ls\sum_{t=1}^T\delta_t\rs^2
      &\leq3D^2\sigma_2^2\sum_{t=1}^Tt^2
      +6\sigma_1^2\sum_{t=1}^Tt^5
      +\ls\sum_{t=1}^T(\sigma_2D\sqrt t+2\sigma_1t^2)\rs^2\\
      &\leq3D^2\sigma_2^2T^3+6\sigma_1^2T^6
      +(\sigma_2DT^{3/2}+2\sigma_1T^3)^2
      \leq10(\sigma_2DT^{3/2}+\sigma_1T^3)^2
      \stackrel{\eqref{eq:optimal_reg}}{\leq}10\lambda_T^2D^4.
    \end{aligned}
  \end{equation}

  For the sum of $\ell_t$, the zero conditional means imply that the cross terms vanish. For $i<j$,
  \begin{equation*}
    \E\lp\ell_i\ell_j\rp
    =\E\lp\ell_i\E\lp\ell_j\mid\mathcal G_j\rp\rp=0.
  \end{equation*}
  Consequently,
  \begin{equation}
    \label{eq:optimal_stochastic_sum_moments}
      \E\ls\sum_{t=1}^T\ell_t\rs^2
      =\sum_{t=1}^T\E\ell_t^2
      \stackrel{\eqref{eq:optimal_ell_second_moment}}{\leq}
      \sigma_1^2\sum_{t=1}^Tt^5
      =\frac{\sigma_1^2T^2(T+1)^2(2T^2+2T-1)}{12}
      \leq\sigma_1^2T^6
      \stackrel{\eqref{eq:optimal_reg}}{\leq}\lambda_T^2D^4.
  \end{equation}

  Substituting~\eqref{eq:optimal_error_sum_moment} and~\eqref{eq:optimal_stochastic_sum_moments} into~\eqref{eq:optimal_energy_second_split}, we obtain
  \begin{equation*}
    \ls\E E_T^2\rs^{1/2}
    \leq\ls\frac{1}{2}+\sqrt{10}+1\rs\lambda_TD^2
    \leq5\lambda_TD^2.
  \end{equation*}
  This proves~\eqref{eq:optimal_energy_moments}.
\end{proof}

\subsection{\texorpdfstring{$A_T$}{A T} Bounds}

We now bound $\ls\E A_T^{-2}\rs^{1/2}$ in~\eqref{eq:goal}. 

For $t\geq2$, we have:
\begin{itemize}
  \item If $a_t=\alpha_tA_t$, then
  \begin{equation*}
    A_t=A_{t-1}+a_t
    =\frac{A_{t-1}}{1-\alpha_t},
    \qquad A_t^{-1}=(1-\alpha_t)A_{t-1}^{-1}.
  \end{equation*}
  \item If $a_t<\alpha_tA_t$, then
  \begin{equation*}
    a_t<\alpha_tA_t=\alpha_t(A_{t-1}+a_t),
    \qquad a_t<\frac{\alpha_tA_{t-1}}{1-\alpha_t}.
  \end{equation*}
  Therefore, the minimum in~\eqref{eq:optimal_weights} gives
  \begin{equation}
    \label{eq:optimal_damped_weight}
    a_t=\frac{\lambda_t}{\alpha_t\tau_t}
    \stackrel{\eqref{eq:optimal_reg},\eqref{eq:optimal_damping}}{=}
    \frac{1}{\alpha_t\ls D/t^{3/2}+\norm{s_t}\rs}.
  \end{equation}
\end{itemize}

Thus, the growth of $A_t=A_{t-1}+a_t$ depends on the step lengths at iterations with $a_t<\alpha_tA_t$. The following lemma bounds $A_T^{-1}$ in terms of these step lengths.

\begin{lemma}
  For every integer $T\geq1$,
  \begin{equation}
    \label{eq:optimal_inverse_step_sum}
    A_T^{-1}\leq\frac{128D}{T^{7/2}}
    +\frac{256}{T^4}
    \sum_{\substack{2\leq t\leq T\\a_t<\alpha_tA_t}}\frac{\norm{s_t}}{\alpha_t}.
  \end{equation}
\end{lemma}

\begin{proof}
  For $t\geq2$ with $a_t<\alpha_tA_t$, we have $0<\alpha_t<1$ and
  \begin{equation*}
    A_t=A_{t-1}+a_t
    =(1-\alpha_t)\frac{A_{t-1}}{1-\alpha_t}
    +\alpha_t\frac{a_t}{\alpha_t}.
  \end{equation*}
  By convexity of $1/x$
  \begin{equation}
    \label{eq:optimal_inverse_step}
    \begin{aligned}
      \frac1{A_t}
      \leq(1-\alpha_t)\frac1{A_{t-1}/(1-\alpha_t)}
      +\alpha_t\frac1{a_t/\alpha_t}
      =\frac{(1-\alpha_t)^2}{A_{t-1}}+\frac{\alpha_t^2}{a_t}
      \stackrel{\eqref{eq:optimal_damped_weight}}{\leq}
      \frac{1-\alpha_t}{A_{t-1}}
      +\frac{D\alpha_t^3}{t^{3/2}}+\alpha_t^3\norm{s_t}.
    \end{aligned}
  \end{equation}
For $t\geq2$ with $a_t<\alpha_tA_t$, multiplying inequality above by $t(t+1)(t+2)(t+3)$ and substituting $\alpha_t=4/(t+3)$ gives
\begin{gather*}
  \frac{t(t+1)(t+2)(t+3)}{A_t}
  -\frac{(t-1)t(t+1)(t+2)}{A_{t-1}}\\
  \leq\frac{64D}{\sqrt t}\frac{(t+1)(t+2)}{(t+3)^2}
  +256\frac{t(t+1)(t+2)}{(t+3)^3}\frac{\norm{s_t}}{\alpha_t}\\
  \leq\frac{64D}{\sqrt t}+256\frac{\norm{s_t}}{\alpha_t}.
\end{gather*}

For $a_t=\alpha_tA_t$, we have equality:
  \begin{equation*}
      \frac1{A_t}
      =(1-\alpha_t)\frac1{A_{t-1}/(1-\alpha_t)}
      +\alpha_t\frac1{a_t/\alpha_t}
      =
      \frac{1-\alpha_t}{A_{t-1}}.
  \end{equation*}

For $t\geq2$ with $a_t=\alpha_tA_t$, multiplying the equality above by the same factor gives
\begin{equation*}
  \frac{t(t+1)(t+2)(t+3)}{A_t}
  -\frac{(t-1)t(t+1)(t+2)}{A_{t-1}}=0.
\end{equation*}
Summing these two relations over $t=2,\ldots,T$, we obtain
\begin{gather*}
  \frac{T(T+1)(T+2)(T+3)}{A_T}
  \leq\frac{24}{A_1}
  +\sum_{\substack{2\leq t\leq T\\a_t<\alpha_tA_t}}
  \ls\frac{64D}{\sqrt t}+256\frac{\norm{s_t}}{\alpha_t}\rs
  \leq\frac{24}{A_1}+64D\sum_{t=2}^T\frac1{\sqrt t}
  +256\sum_{\substack{2\leq t\leq T\\a_t<\alpha_tA_t}}\frac{\norm{s_t}}{\alpha_t}\\
  \leq48D+128D(\sqrt T-1)
  +256\sum_{\substack{2\leq t\leq T\\a_t<\alpha_tA_t}}\frac{\norm{s_t}}{\alpha_t}
  \leq128D\sqrt T
  +256\sum_{\substack{2\leq t\leq T\\a_t<\alpha_tA_t}}\frac{\norm{s_t}}{\alpha_t}.
\end{gather*}
Here we used
\begin{equation}
  \label{eq:optimal_inverse_sqrt_sum}
  A_1^{-1}\stackrel{\eqref{eq:optimal_initialization}}{=}\frac{\tau_1}{\lambda_1}
  =D+\norm{s_1}\stackrel{\eqref{eq:optimal_geometry}}{\leq}2D \quad \text{and} \quad \sum_{t=1}^T\frac1{\sqrt t}\leq2\sqrt T-1.
\end{equation}
Dividing by $T(T+1)(T+2)(T+3)\geq T^4$ proves~\eqref{eq:optimal_inverse_step_sum}.
\end{proof}

By~\eqref{eq:optimal_geometry}, each term $\norm{s_t}/\alpha_t$ is bounded by $D$, so the sum in~\eqref{eq:optimal_inverse_step_sum} has finite second moment. Also, $A_T^{-1}\leq A_1^{-1}\leq2D$. Applying Minkowski's inequality in~\eqref{eq:optimal_inverse_step_sum}, we obtain
\begin{equation}
  \label{eq:optimal_inverse_reduction}
  \ls\E A_T^{-2}\rs^{1/2}
  \leq\frac{128D}{T^{7/2}}
  +\frac{256}{T^4}
  \ls\E\ls\sum_{\substack{2\leq t\leq T\\a_t<\alpha_tA_t}}
  \frac{\norm{s_t}}{\alpha_t}\rs^2\rs^{1/2}.
\end{equation}
To obtain $\ls\E A_T^{-2}\rs^{1/2}=\cO(D/T^{7/2})$, it suffices to bound the second moment of the sum on the right by $\cO(D^2T)$. We use the squared-step term retained in~\eqref{eq:optimal_one_step}. At iterations with $a_t<\alpha_tA_t$, we have $a_t\alpha_t\tau_t=\lambda_t$, so we define
\begin{equation}
  \label{eq:optimal_movement}
  V_t=\frac{a_t\tau_t}{\alpha_t\lambda_t}\norm{s_t}^2.
\end{equation}
By~\eqref{eq:optimal_damped_weight}, $\norm{s_t}/\alpha_t=\sqrt{V_t}$ whenever $a_t<\alpha_tA_t$. Therefore,
\begin{equation*}
  \sum_{\substack{2\leq t\leq T\\a_t<\alpha_tA_t}}\frac{\norm{s_t}}{\alpha_t}
  =\sum_{\substack{2\leq t\leq T\\a_t<\alpha_tA_t}}\sqrt{V_t}
  \leq\sum_{t=1}^T\sqrt{V_t}.
\end{equation*}
It therefore suffices to bound $\E\ls\sum_{t=1}^T\sqrt{V_t}\rs^2$.

\begin{lemma}
  \label{lem:optimal_movement}
  For every integer $T\geq1$,
  \begin{equation}
    \label{eq:optimal_movement_bounds}
    \E\sum_{t=1}^T\sqrt t V_t\leq76D^2\sqrt T,
    \qquad
    \E\ls\sum_{t=1}^T\sqrt{V_t}\rs^2\leq152D^2T.
  \end{equation}
\end{lemma}

\begin{proof}
  Dividing~\eqref{eq:optimal_one_step} by $\lambda_t$, taking expectations, and using $\E\ell_t=0$, we get
  \begin{equation}
    \label{eq:Vt_bound}
    \frac18\E V_t
    \leq\frac{\E E_{t-1}}{\lambda_t}-\frac{\E E_t}{\lambda_t}
    +\frac{D^2}{2}\ls1-\frac{\lambda_{t-1}}{\lambda_t}\rs
    +\frac{\E\delta_t}{\lambda_t}.
  \end{equation}
  We bound the third term first. By the definition of $\lambda_t$ in~\eqref{eq:optimal_reg}, for $t\geq2$,
  \begin{gather*}
    \lambda_{t-1}
    =\frac{\sigma_1t^3}{D^2}\ls1-\frac1t\rs^3
    +\frac{\sigma_2t^{3/2}}{D}\ls1-\frac1t\rs^{3/2}+2M\\
    \geq\ls1-\frac1t\rs^3
    \ls\frac{\sigma_1t^3}{D^2}+\frac{\sigma_2t^{3/2}}{D}+2M\rs
    =\ls1-\frac1t\rs^3\lambda_t.
  \end{gather*}
  Since $\lambda_t\geq\lambda_{t-1}$,
  \begin{equation*}
    0\leq1-\frac{\lambda_{t-1}}{\lambda_t}
    \leq\frac3t-\frac3{t^2}+\frac1{t^3}\leq\frac3t.
  \end{equation*}
  The last bound also holds at $t=1$ because $\lambda_0=\lambda_1$. 
  
  Next, we bound the last term in~\eqref{eq:Vt_bound}
  \begin{equation*}
    \frac{\E\delta_t}{\lambda_t}
    \stackrel{\eqref{eq:optimal_error_budget} }{\leq}\frac{\sigma_2D\sqrt t+2\sigma_1t^2}{\lambda_t}
    \leq\frac{2D^2}{t}
    \frac{\sigma_2t^{3/2}/D+\sigma_1t^3/D^2}{\lambda_t}
    \stackrel{\eqref{eq:optimal_reg}}{\leq}\frac{2D^2}{t}.
  \end{equation*}
  Substituting these bounds into~\eqref{eq:Vt_bound} and using $\lambda_t\geq\lambda_{t-1}$ gives
  \begin{equation}
    \label{eq:optimal_normalized_step}
    \frac18\E V_t
    \leq\frac{\E E_{t-1}}{\lambda_{t-1}}-\frac{\E E_t}{\lambda_t}
    +\frac{7D^2}{2t}.
  \end{equation}
  Multiplying~\eqref{eq:optimal_normalized_step} by $\sqrt t$ and summing gives
  \begin{gather*}
    \frac18\E\sum_{t=1}^T\sqrt t V_t
    \leq\sum_{t=1}^T\sqrt t\ls
      \frac{\E E_{t-1}}{\lambda_{t-1}}-\frac{\E E_t}{\lambda_t}\rs
    +\frac{7D^2}{2}\sum_{t=1}^T\frac1{\sqrt t}\\
    =\frac{E_0}{\lambda_0}
    +\sum_{t=1}^{T-1}(\sqrt{t+1}-\sqrt t)\frac{\E E_t}{\lambda_t}
    -\sqrt T\frac{\E E_T}{\lambda_T}
    +\frac{7D^2}{2}\sum_{t=1}^T\frac1{\sqrt t}\\
    \stackrel{\eqref{eq:optimal_energy_moments},\eqref{eq:optimal_inverse_sqrt_sum}}{\leq}
    \frac{D^2}{2}+\frac{5D^2}{2}(\sqrt T-1)
    +\frac{7D^2}{2}(2\sqrt T-1)
    =\frac{D^2}{2}(19\sqrt T-11)
    \leq\frac{19D^2}{2}\sqrt T.
  \end{gather*}
  This proves the first bound in~\eqref{eq:optimal_movement_bounds}. By Cauchy--Schwarz,
  \begin{gather*}
    \ls\sum_{t=1}^T\sqrt{V_t}\rs^2
    =\ls\sum_{t=1}^Tt^{-1/4}\ls t^{1/4}\sqrt{V_t}\rs\rs^2
    \leq\ls\sum_{t=1}^T\frac1{\sqrt t}\rs
          \ls\sum_{t=1}^T\sqrt t V_t\rs,\\
    \E\ls\sum_{t=1}^T\sqrt{V_t}\rs^2
    \stackrel{\eqref{eq:optimal_inverse_sqrt_sum}}{\leq}
    (2\sqrt T-1)\E\sum_{t=1}^T\sqrt t V_t
    \leq(2\sqrt T-1)76D^2\sqrt T
    \leq152D^2T.
  \end{gather*}
\end{proof}

We restate Theorem~\ref{thm:optimal_convergence} with explicit constants.

\begin{theorem}
  Let $f$ be convex and satisfy Assumption~\ref{as:lip_hess}. Let $g_t$ and $H_t$ satisfy Assumptions~\ref{as:stoch_grad} and~\ref{as:stoch_hess}, respectively. Let $\widetilde g_t$ satisfy Assumptions~\ref{as:stoch_grad} and~\ref{as:unbiased_grad}. Then Algorithm~\ref{alg:optimal_method_appendix} with $M\geq L_2$ and parameters~\eqref{eq:optimal_reg} satisfies, for every integer $T\geq1$,
  \begin{equation*}
    \E\lp f(w_T)-f^*\rp
    \leq640\ls1+4\sqrt{38}\rs
    \ls\frac{\sigma_1D}{\sqrt T}
    +\frac{\sigma_2D^2}{T^2}+\frac{2MD^3}{T^{7/2}}\rs.
  \end{equation*}
\end{theorem}

\begin{proof}
  Substituting~\eqref{eq:optimal_movement_bounds} into~\eqref{eq:optimal_inverse_reduction} gives
  \begin{equation}
    \label{eq:optimal_inverse_bounds}
    \ls\E A_T^{-2}\rs^{1/2}
    \leq\frac{128D}{T^{7/2}}
    +\frac{512\sqrt{38}D\sqrt T}{T^4}
    =128\ls1+4\sqrt{38}\rs\frac{D}{T^{7/2}}.
  \end{equation}

  Combining~\eqref{eq:optimal_energy_moments} and~\eqref{eq:optimal_inverse_bounds} in~\eqref{eq:goal}, we conclude that
  \begin{gather*}
    \E\lp f(w_T)-f^*\rp
    \leq640\ls1+4\sqrt{38}\rs\frac{\lambda_TD^3}{T^{7/2}}
    \stackrel{\eqref{eq:optimal_reg}}{=}
    640\ls1+4\sqrt{38}\rs
    \ls\frac{\sigma_1D}{\sqrt T}
    +\frac{\sigma_2D^2}{T^2}+\frac{2MD^3}{T^{7/2}}\rs.
  \end{gather*}
\end{proof}

\clearpage
\section{Experimental Details}

\textbf{Hyperparameters for Figure~\ref{fig:ball_reruns}.}
On \texttt{a9a}, the selected values are $M=1$ and
$\sigma_1=0.1$ for both Hessian batch sizes 50 and 1,000. On
\texttt{covtype}, the selected values are $M=0.01$ and $\sigma_1=0.01$ for
both Hessian batch sizes 50 and 1,000.

\textbf{Computing infrastructure.}
All computations use double precision. Hyperparameter sweeps were tuned using NVIDIA A100 GPU.

\end{document}